\documentclass[12pt,reqno]{amsart}
\usepackage{amssymb,latexsym,amsmath,amsthm,amscd,epsfig}
\usepackage[latin1]{inputenc}
\usepackage{pdfsync}

\def\cal{\mathcal}

\def\Re{\text{\rm Re\,}}

\def\cS{\mathcal{S}}
\def\cG{\mathcal{G}}
\def\cP{\mathcal{P}}

\def\bpm{\begin{pmatrix}}
\def\epm{\end{pmatrix}}

\newcommand{\rf}[1]{{\eqref{#1}}}

\newcommand{\fiproof}{{\hspace*{\fill} $\square$ \vspace{2pt}}}
\newcommand{\D}{{\mathbb D}}  
\newcommand{\wt}[1]{{\widetilde{#1}}}

\newcommand{\pv}{{\rm p.v.}}

\newtheorem{thm}{Theorem}[section]

\newtheorem{proposition}[thm]{Proposition}

\newtheorem{lemma}[thm]{Lemma}

\newtheorem{theorem}[thm]{Theorem}

\newtheorem{theo}[thm]{Theorem}
\newtheorem{lem}[thm]{Lemma}
\newtheorem{coro}[thm]{Corollary}

\theoremstyle{remark}
\newtheorem{remark}[thm]{\bf Remark}

\numberwithin{equation}{section}

\def\R{\mathbb R}
\def\C{\mathbb C}
\def\D{\mathbb D}
\def\H{\mathcal H}

\def\cB{\mathcal B}
\def\cG{\mathcal G}
\def\M{\mathcal M}

\def\S{\mathcal S}

\def\Z{\mathbb Z}
\def\diam{\text{diam}}

\def\supp{\operatorname{supp}}
\def\diam{\operatorname{diam}}

\newcommand{\ve}{{\varepsilon}}

\title[Quasiconformal distortion of capacities and Hausdorff measures]{Quasiconformal distortion of Riesz capacities and Hausdorff measures in the plane}

\author{K. Astala}
\address{Department of Mathematics and Statistics, University of Helsinki, FI-00014 Helsinki, Finland}
\email{\tt astala@mappi.helsinki.fi}

\author{A. Clop}
\address{Departament de Matem\`{a}tiques,  Universitat Aut\`{o}noma de Barcelona, 08193 Bellaterra (Barcelona), Catalunya}
\email{{\tt albertcp@mat.uab.cat}}
\urladdr{http://mat.uab.cat/~albertcp}

\author{X. Tolsa}
\address{Instituci\'{o} Catalana de Recerca i Estudis Avan\c{c}ats (ICREA) and Departament de Matem\`{a}tiques,  Universitat Aut\`{o}noma de Barcelona, 08193 Bellaterra (Barcelona), Catalunya}
\email{{\tt xtolsa@mat.uab.cat}}
\urladdr{http://mat.uab.cat/~xtolsa}

\author{I. Uriarte-Tuero}
\address{Department of Mathematics, Michigan State University, East Lansing, MI 48824, USA}
\email{{\tt ignacio@math.msu.edu}}

\author{J. Verdera}
\address{Departament de Matem\`{a}tiques,  Universitat Aut\`{o}noma de Barcelona, 08193 Bellaterra (Barcelona), Catalunya}
\email{{\tt jvm@mat.uab.cat}}



\begin{document}

\begin{abstract}
In this paper we prove the sharp distortion estimates for the quasiconformal
mappings in the plane, both in terms of the Riesz capacities from non linear potential theory and in terms of the Hausdorff measures.
\end{abstract}

\maketitle

\thispagestyle{empty}



\section{Introduction}
\noindent
A $K$-quasiconformal mapping is an orientation preserving homeomorphism $f:\Omega\to\Omega'$ between domains $\Omega,\Omega'\subset\R^n$ that belongs to the Sobolev space $W^{1,n}_{loc}(\Omega;\Omega')$ and satisfies the distortion inequality
\begin{equation}
\max_{|\xi|=1}|\partial_\xi f(x)|\leq K\,\min_{|\xi|=1}|\partial_\xi f(x)|
\end{equation}
at almost every point $x\in\Omega$. If $K=1$, then $f$ is indeed a conformal mapping. If one does not require $f$ to be a homeomorphism, then we simply say that $f$ is $K$-quasiregular. For more background on these mappings, see the monograph \cite{astalaiwaniecmartin}.\\
\\
In the planar setting ($n=2$), Astala's Theorem \cite{astalaareadistortion} solved the long standing Gehring-Reich conjecture on the area distortion of quasiconformal mappings, namely
\begin{equation}\label{area}
|\phi(E)|\leq C_K\,|E|^{1/K},
\end{equation}
where $\phi:\Omega\to\Omega'$ is a conveniently normalized $K$-quasiconformal mapping between planar domains, and $E$ is a measurable subset of $\Omega$. Other related questions, like the optimal integrability and the sharp Hausdorff dimension distortion, were solved as a consequence of \eqref{area}. In particular, one has
\begin{equation}\label{dimension}
\frac{1}{\dim\phi(E)}-\frac12\leq\frac{1}{K}\left(\frac{1}{\dim E}-\frac12\right).
\end{equation}
Furthermore, in \cite{astalaareadistortion} it was shown that equality can be attained for some sets $E$ and mappings $\phi$.\\
\\
In the last years renewed interest has arisen in connection with these questions, and deep advances have been made, improving the above inequality in several directions. The sharp
 quasiconformal distortion of Hausdorff contents $\M^ t$, $0<t<2$, was obtained by Lacey, Sawyer and Uriarte-Tuero \cite{Lacey-Sawyer-Uriarte}. They showed that if $E$ is contained in some ball $B$, $0<t<2$, and $t'=\frac{2Kt}{2+(K-1)t}$, then
\begin{equation}\label{contents}
\frac{\M^{t'}(\phi(E))}{\diam(\phi(B))^{t'}} \leq C(K) \,\left(\frac{\M^t(E)}{\diam(B)^t}\right)^{\frac{t'}{Kt}},
\end{equation}
which in particular proves the following implication about the corresponding Hausdorff measures
$\H^t$:
\begin{equation}\label{zerotozero}
 \H^t(E)=0\hspace{1cm}\Rightarrow\hspace{1cm}\H^{t'}(\phi(E))=0.
\end{equation}
This extends \eqref{dimension}, and answers in the affirmative a conjecture by Astala \cite{astalaareadistortion}. Previously, in \cite{ACMOU}, the particular case $t'=1$ had been solved. In any case, notice that from \rf{contents} it is not clear if $\H^{t'}(\phi(E))<\infty$ whenever 
$\H^t(E)<\infty$.\\
\\
The optimal quasiconformal distortion of analytic capacity has also been a topic of deep research (see for instance \cite{astalaareadistortion} or \cite{ACMOU}). In a recent joint work of Tolsa and Uriarte-Tuero \cite{Tolsa-Uriarte}, it is shown that, for $K>1$,
\begin{equation}\label{gammavsriesz}
\gamma(\phi(E))\leq C_K\,\,\left(\dot {\cal C}_{\frac{2K}{2K+1},\frac{2K+1}{K+1}}(E)\right)^\frac{K+1}{2K},
\end{equation}
where $\dot{\cal C}_{\alpha,p}$ is the classical Riesz capacity of nonlinear potential theory (see \rf{defcap}), and $\gamma$ denotes the analytic capacity. This estimate has remarkable consequences in the determination of removable sets for bounded $K$-quasiregular mappings. For the holomorphic case, see \cite{tolsasemiadditivityanalyticcapacity} and \cite{davidunrectifiable1setszeroanalyticcapacity}. To get \eqref{gammavsriesz}, the authors first show a sharp bound for the distortion of a Hausdorff content $\M^h$ (see \cite[Lemma 2.11]{Tolsa-Uriarte}), where $h$ is a gauge function which
is not invariant under translations. As a matter of fact, it turns out that 
Riesz capacities can be recovered as a supremum of Hausdorff contents $\M^h$ with $h$ running within some precise class. This allowed the authors to prove a second estimate, now concerning quasiconformal distortion of Riesz capacities. More precisely, for each $q>1$, they showed that
\begin{equation}\label{1dimensriesz}
\dot {\cal C}_{\frac1q,q}(\phi(E))\leq C_{K,q}\,\,\dot {\cal C}_{\alpha,p}(E)^\frac{K+1}{2K}
\end{equation}
where $p=1+\frac{2K}{K+1}(q-1)$ and $2-\alpha p=\frac{2}{K+1}$. Note that the left hand side is a $1$-dimensional quantity for every $q>1$. 
\\
\\
In the present paper, we extend \eqref{1dimensriesz} to  all other indices $\alpha,p$ and obtain a general version for the quasiconformal distortion of all Riesz capacities.

\begin{theorem} \label{teocap2}
Let $1<q<\infty$ and $0<\beta q<2$. Let $t'=2-\beta q$, and $t$ be such that
$$
\frac1t-\frac12 = K\left(\frac{1}{t'}-\frac12\right).
$$
Let $E\subset\C$ be compact, and let $\phi:\C\to\C$ be a $K$-quasiconformal map. If $E$ is contained in a ball $B$, then
\begin{equation}\label{eq11}
\frac{\dot{{\cal C}}_{\beta,q}(\phi(E))}{\diam(\phi(B))^{t'}}\leq C(\beta,q,K)\,\left(\frac{\dot{{\cal C}}_{\alpha,p}(E)}{\diam(B)^t}\right)^\frac{t'}{Kt}
\end{equation}
where
$$
p=1+\frac{Kt}{t'}\,(q-1)\hspace{1cm}\text{and}\hspace{1cm} 2-\alpha p=t.
$$
The constant in \rf{eq11} depends only on $\beta$, $q$, $K$.
\end{theorem}

\noindent
This result follows by combining some of the ideas from \cite{Tolsa-Uriarte} with others from
 \cite{Lacey-Sawyer-Uriarte}. Following this approach, in Lemma \ref{lemkgran} below one obtains distortion estimates in terms of the $h$-contents $\M^h$, with $h$ of the form $h(B(x,r))=r^t\,
 \ve(B(x,r)),$ with $0<t<2$ and $\ve(\cdot)$ satisfying some appropriate conditions. 
This result extends \cite[Lemma 2.11]{Tolsa-Uriarte} (which only dealt with the case $t'=1$). 
Theorem \ref{teocap2} is a direct consequence of the distortion estimates in terms of $h$-contents.
\\
\\
A second main result that we establish, using the $h$-contents $\M^h$, is the following distortion theorem involving Hausdorff measures.


\begin{theorem} \label{teopri}
Let $0<t<2$ and denote $t' =\frac{2Kt}{2+(K-1)t}$. Let $\phi:\C\to\C$ be $K$-quasiconformal. For any ball $B$ and any compact set $E\subset B$, we have
 \begin{equation}\label{eqhaus}
\frac{\H^{t'}(\phi(E))}{\diam(\phi(B))^{t'}} \leq C(K) \,\left(\frac{\H^t(E)}{\diam(B)^t}\right)^{\frac{t'}{Kt}}.
\end{equation}
In particular, if $\H^t(E)$ is finite, then also $\H^{t'}(\phi(E))$ is finite.
\end{theorem}

\noindent
Notice that \rf{eqhaus} is the estimate that one gets replacing Hausdorff contents by Hausdorff measures in \rf{contents}. 
This result may seem somewhat surprising, because in the arguments used in 
\cite{Lacey-Sawyer-Uriarte} (and in \cite{ACMOU} in the case $t'=1$) to prove \rf{contents} it 
is essential that one works with Hausdorff contents, and not with Hausdorff 
measures: for instance, many estimates
in \cite{Lacey-Sawyer-Uriarte} and \cite{ACMOU} involve some packing conditions which hold for Hausdorff
contents but not for Hausdorff measures.\\
\\
Let us remark that Theorem \ref{teopri} was proved in \cite{Tolsa-preprintqc}. However, since this result follows
using the distortion estimates for $h$-contents in Lemma \ref{lemkgran}, X. Tolsa has preferred to include this in the present paper (see Section \ref{tolsapreprint}).\\
\\
 An immediate corollary of Theorem \ref{teopri} is the following. 

\begin{coro}\label{teosigma}
Let $E\subset\C$ be compact and $\phi:\C\to\C$ a $K$-quasiconformal map. 
If $\H^t(E)$ is $\sigma$-finite for some $0<t<2$, then $\H^{t'}(\phi(E))$ is $\sigma$-finite for $t'=\frac{2Kt}{2+Kt-t}$.
\end{coro}

As usual, the letters $c,C$ denote constants (often, absolute constants) that may change 
at different occurrences, while constants with a subscript, such as $C_1$, 
retain their values. The notation $A\lesssim B$ means that there is a positive constant $C$ such that $A\leq CB$, and $A\simeq B$ means that $A\lesssim B\lesssim A$.

\bigskip


\section{Measures, gauge functions and Hausdorff contents}  \label{sec1}

\subsection{Strategy for the proof of Theorem \ref{teocap2}}\label{strategy}
\noindent
To motivate the introduction of the $h$ contents below, we will describe 
the main ideas in the proof of Theorem \ref{teocap2}. Recall that the homogeneous Riesz capacity $\dot{\cal C}_{\alpha,p}$ is defined as
\begin{equation}\label{defcap}
\dot{\cal C}_{\alpha,p}(E)=\inf\{\|g\|_{L^p(\C)}; g\in L^p(\C),\,I_\alpha\ast g\geq\chi_E   \},
\end{equation}
where $I_\alpha(z)=\frac{1}{|z|^{2-\alpha}}$ is the usual planar Riesz kernel of order $\alpha$. Further, by Wolff's Theorem (see for instance \cite{adamshedberg}), it turns out that
$$
\dot{\cal C}_{\alpha,p}(E)\simeq\sup\{\mu(E); \supp(\mu)\subset E, \dot{W}^{\mu}_{\alpha,p}(z)\leq 1\,\,\forall z\in\C\}
$$ 
where 
$$\dot{W}^\mu_{\alpha,p}(z)=\int_0^\infty\left(\frac{\mu(B(z,r))}{r^{2-\alpha p}}\right)^{p'-1}\frac{dt}{t}$$
is the homogeneous Wolff potential of $\mu$. In this paper, we prove that $\dot{\cal C}_{\alpha,p}(E)$ coincides with the following supremum of generalized Hausdorff contents, modulo multiplicative constants:
$$
\dot{{\cal C}}_{\alpha,p}(E)
\simeq\sup\left\{\!\M^h(E): h(x,r)=r^{2-\alpha\,p}\,\ve(x,r),\;\ve\in{\cal G}_2,\int_0^\infty\!\left(\frac{h(x,r)}{r^{2-\alpha\,p}}\right)^{p'-1}\frac{dr}{r}\leq 1\!
\right\}
$$
(see Subsection \ref{Gp} for the definition of ${\cal G}_2$). Then the problem is reduced to see how quasiconformal mappings distort the generalized Hausdorff contents $\M^h$ whenever $h$ is an admissible gauge function for $\dot{{\cal C}}_{\alpha,p}$. 
In particular, to each finite Borel measure $\mu$ with bounded Wolff potential, we can associate a gauge function $h=h_\mu$ admissible for the Riesz capacity. Actually, in
 the supremum above one may restrict to such gauge functions $h_\mu$. This fact is very useful because the Hausdorff measures $\H^{h_\mu}$ and contents $\M^{h_\mu}$ can be seen as \emph{regularized} versions of $\mu$, whence easier to work with. Such gauge functions were already introduced in \cite{Tolsa-Uriarte}, but for the reader's convenience we remember their definition and main properties in the next subsections.

\noindent

\subsection{The gauge functions $h_{\mu,a,t}$}
Let $0<t<2$ and $a>0$ be fixed parameters. We consider the function 
\begin{equation}\label{eqpsia}
\psi_{a,t}(x) = \frac1{|x|^{t+a}+1},\qquad x\in\C.
\end{equation}
Given a compactly supported finite Borel measure $\mu$, let us define for every ball $B=B(x,r)$ the gauge function
\begin{equation}\label{defhx}
h(x,r)=h_{\mu,a,t}(x,r) = h_{\mu,a,t}(B) =\int \psi_{a,t}\left(\frac{|x-y|}{r}\right)\,d\mu(y).
\end{equation}
Now, for our fixed $a,t$, set
\begin{equation}\label{defvex}
\ve(x,r)=\ve_{\mu,a,t}(x,r)= \frac{h_{\mu,a,t}(x,r)}{r^ t}=\frac1{r^t} \int \psi_{a,t}\Bigl(\frac{y-x}r\Bigr)d\mu(y),
\end{equation}
so that $h_{\mu,a,t}(x,r)=r^t\,\ve_{\mu,a,t}(x,r)$. One should view $t$ as a dimensional parameter,
while the role of $a$ is to provide enough decay at $\infty$ of $\psi_{a,t}$.
Notice that, by construction, 
$\mu(B)\leq 2\,h_{\mu,a,t}(B)$, and that $h_{\mu,a,t}(B)$ can be seen as a 
smooth version of $\mu(B)$. Similarly, $\ve_{\mu,a,t}(B)$ is a kind of smooth substitute of the
 $t$-dimensional density $\theta_\mu^t(B)=\mu(B)/r(B)^t$. One of the advantages of 
$\ve_{\mu,a,t}(x,r)$ over $\theta_\mu^t(B(x,r))$ is that, for $C=2^a$, 
$$\ve_{\mu,a,t}(x,2r)\leq C\ve_{\mu,a,t}(x,r)$$ 
for any $x$ and $r>0$, which fails in general for $\theta_\mu^t(x,r)$. Analogously, we have 
$$h_{\mu,a,t}(x,2r)\leq C\,h_{\mu,a,t}(x,r),$$ 
for $C=2^{at}$,
while  $\mu(B(x,r))$ and $\mu(B(x,2r))$ may be very different.

\bigskip


\subsection{The measures $\H^h$ and the contents $\M^h$}\label{Hh} Let $\cB$ denote the family of all closed balls contained in $\C$, and let $\ve:\cB\to[0,\infty)$ be any function defined on $\cB$. We set $\ve(x,r)=\ve(B(x,r))$, and we define $h(x,r)=\ve(x,r)\,r^t$.  We assume that 
$$\lim_{r\to 0}h(x,r)=0\hspace{1cm}\text{ for all }x\in\C.$$
We introduce the measure $\H^h$ following Carath\'eodory's construction (see \cite{mattila}, p.54): given $0<\delta\leq\infty$ and a set $F\subset\C$,
$$ \H^h_\delta(F) = \inf\sum_i h(B_i),$$
where the infimum is taken over all coverings $F\subset \bigcup_i B_i$ with balls $B_i$ with radii smaller that $\delta$. Finally, we
define 
$$ \H^h(F) = \lim_{\delta\to0}  \H^h_\delta(F).$$
The above limit exists, because $\H^h_\delta(F)$ is a non-increasing function of $\delta$. For $\delta=\infty$, we obtain the $h$-content, and we simply write $\M^h(E)=\H_\infty^h(E)$. Recall also that $ \H^h$ is a Borel regular measure (see \cite{mattila}), although it is not a ``true'' Hausdorff measure. It is clear that $\M^h(F)\leq\H^h(F)$. On the other hand, the implication
$$\M^h(F)=0\hspace{1cm}\Rightarrow\hspace{1cm}\H^h(F)=0$$
also holds if the function $r\mapsto h(x,r)$ is non-decreasing for all $x\in F$. This is the case, for instance, if $h=h_{\mu,a,t}$ for certain measure $\mu$. 

\begin{lemma}\label{Mmu}
For any Borel set $A$, we have
$$\mu(A)\leq 2\,\M^{h_{\mu,a,t}}(A).$$
\end{lemma}
\begin{proof}
Given any $\eta>0$, consider a covering $A\subset \bigcup_i B_i$ by balls so that
$$\sum_i h_{\mu,a,t}(B_i) \leq   \M^{h_{\mu,a,t}}(A) +\eta.$$ 
Since $\mu(B_i)\leq 2 h_{\mu,a,t}(B_i)$, we have
$$\mu(A)\leq \sum_i \mu(B_i)\leq 2\sum_i h_{\mu,a,t}(B_i) \leq 2 \M^{h_{\mu,a,t}}(A) + C\eta.$$ 
\end{proof}


\subsection{The families $\cG_1$ and $\cG_2$} 
\label{Gp}. We say that the function $\ve:\cB\to[0,\infty)$ belongs to $\cG_1$ if there exists a constant $C_2$ such that 
\begin{equation}\label{eqeq1}
C_2^{-1}\,\ve(x,r)\leq \ve(y,s)\leq C_2\,\ve(x,r)
\end{equation}
whenever $|x-y|\leq 2r$ and $r/2\leq s\leq 2r$. Note that \rf{eqeq1} also holds with a different constant $C_2$ if one assumes $|x-y|\leq Cr$ and $C^{-1}r\leq s\leq Cr$, by applying \rf{eqeq1} finitely many
times. It is easy to check that every $\ve_{\mu,a,t}$ belongs to $\cG_1$, due to the properties of the function $\psi_{a,t}$.\\
\\
\noindent
It was noticed in \cite{Tolsa-Uriarte} that if $\ve\in{\cal G}_1$ then Frostman's Lemma holds for $\M^h$, where $h(x,r)=r^t\,\ve(x,r)$:

\begin{lemma}\label{frostman}
If $\ve\in\cG_1$ and $h(x,r)=r^t\ve(x,r)$, then Frostman's Lemma holds for $ \M^h$. That is,
given a compact set $F\subset \C$, then $ \M^h(F)>0$ if and only if
there exists a Borel measure $\nu$ supported on $F$ such that $\nu(B)\leq h(B)$ for any ball $B$. Furthermore, we can find $\nu$ so that $\nu(F)\geq c^{-1}  \M^h(F)$.
\end{lemma}
\noindent
The proof is almost the same as the one of the usual Frostman's Lemma (for instance, see \cite{mattila}, p.112), taking into account the regularity properties of the gauge functions $h\in\cG_1$.\\
\\
Now we introduce the class $\cG_2$. For each fixed $0<t<2$, the class $\cG_2=\cG_2(t)$ consists on functions $\ve=\ve(x,r)$ such that
\begin{equation}\label{eqsum4}
\sum_{k\geq0} 2^{-k(2-t)}\,\ve(x,2^kr)\leq C_3 \,\ve(x,r), \text{ for all }x\in\C,
\end{equation}
for certain constant $C_3>0$. Notice that \eqref{eqsum4} is equivalent to saying that
$$\int_r^\infty\frac{\varepsilon(x,s)}{s^{2-t}}\,\frac{ds}{s}\leq C_3\,\frac{\varepsilon(x,r)}{r^{2-t}},$$
or in terms of $h(x,r)=r^t\,\ve(x,r)$,
$$\int_r^\infty\frac{h(x,s)}{s^{2}}\,\frac{ds}{s}\leq C_3\,\frac{h(x,r)}{r^{2}}.$$
Observe that this estimate does not hold for the area, that is, for $h(x,r)=r^2$, neither for gauges $h$ too close to $r^2$, like $h(r)=r^2\log\frac1r$. 

\begin{lemma}\label{emug2}
Let $a>0$, $\alpha,\beta>0$ and $m=\min(\alpha,\beta)$. If $\alpha\neq \beta$, then for all $x>0$ we have
\begin{equation}\label{lemtec1}
 \sum_{k\geq 0} \frac{2^{-\beta k}}{\bigl(2^{-k}x\bigr)^{\alpha} + 1}\leq \frac C{x^{m} + 1},
\end{equation}
with $C$ depending only on $\alpha, \beta$. As a consequence, if $\mu$ is a finite Borel measure, then $\ve_{\mu,a,t}\in \cG_2(t)$ whenever $0<a<2-t$.
\end{lemma}

\begin{proof}
The estimate \eqref{lemtec1} is just a numerical inequality which can be proved by splitting the
sum according to whether $2^{-k}x\geq 1$ or $2^{-k}x< 1$, and then approximating the denominator
inside the sum by $2^{-k}x$ in the first case and by $1$ in the second, for instance. 
We skip the details. 
To deal with the last
statement, we just have to combine \eqref{lemtec1} with the definitions above to get
$$\aligned
\sum_{k=0}^\infty \frac{\ve_{\mu,a,t}(x,2^kr)}{2^{(2-t)k}}
&=r^{-t}\sum_{k=0}^\infty \frac{h_{\mu,a,t}(x,2^kr)}{2^{2k}}= r^{-t}\,\int \sum_{k=0}^\infty 2^{-2k}\,\psi_{a,t}\left(2^{-k}\,\frac{|z-x|}{r}\right)\,d\mu(z)\\
&\leq C\,r^{-t}\,\int \psi_{a,t}\left(\frac{|z-x|}{r}\right)d\mu(z)=C\,r^{-t}\,h_{\mu,a,t}(x,r)=C\,\ve_{\mu,a,t}(x,r),
\endaligned$$
where \eqref{lemtec1} was used with $\alpha=t+a$, $\beta=2$ (hence $m=t+a$). In particular, $\ve_{\mu,a,t}\in\cG_2(t)$. 
\end{proof}

\begin{remark}
If $a=2-t$, then we cannot ensure that $\ve_{\mu,a,t}\in \cG_2(t)$. Indeed, if
 we set $\alpha=\beta$ in the left hand side \eqref{lemtec1}, then one gets a worse estimate.
 One easily checks that in this case,
$$\sum_{k\geq 0} \frac{2^{-\alpha k}}{\bigl(2^{-k}x\bigr)^{\alpha} + 1}
\simeq\frac{1}{\alpha}\,\frac{\log(1+x^\alpha)}{x^\alpha+1},$$
with absolute constants. Hence a logarithmic term appears, which implies that 
$$
\aligned
\sum_{k=0}^\infty \frac{\ve_{\mu,2-t,t}(x,2^kr)}{2^{(2-t)k}}&= r^{-t}\,\int \sum_{k=0}^\infty \frac{2^{-2k}}{\left(2^{-k}\,\frac{|z-x|}{r}\right)^2+1}\,d\mu(z)\\
&\simeq r^{-t} \int\frac{1}{\left(\frac{|z-x|}{r}\right)^2+1}\,\log\left(1+\left(\frac{|z-x|}{r}\right)^2\right)\,d\mu(z)
\endaligned
$$
and so we cannot infer that $\ve_{\mu,2-t,t}\in\cG_2(t)$.
\end{remark}



\subsection{Behavior of $\cG_1$ and $\cG_2$ under quasiconformal mappings}


Need will arise of evaluating gauge functions $h=h(B(x,r))=r^t\,\ve(x,r)$ on sets that are not necessarily balls. To do this, given an arbitrary bounded set $A\subset\C$, let $B$ a ball with minimal diameter that contains $A$. Then we set 
$$\ve(A)=\ve(B).$$
If may happen that $B$ is not unique, but this does not cause any harm. In this case, for definiteness, we can define $\ve(A)$ as the infimum of the values $\ve(B)$ over all balls $B$ with minimal diameter containing $A$. Analogously, if $h(x,r)=r^t\,\ve(x,r)$, we define $h(A)$ as the infimum the $h(B)$'s. \\
Our next objective consists in showing that if $\phi$ is a $K$-quasiconformal planar homeomorphism and $0\leq d\leq1$, then the function defined by
$$\ve(B) = \ve_{\mu,a,t}(\phi(B))^d$$ 
for any ball $B\subset\C$, also belongs to $\cG_1\cap\cG_2$. In fact, because of the geometric properties of quasiconformal mappings and the smoothness of $\psi_{a,t}$, it is easily seen that $\ve$ satisfies \rf{eqeq1}. To show that \rf{eqsum4} also holds requires some more effort.

\begin{lemma}\label{lemtec5}
Let $\ve_{\mu,a,t}$ be as above, and let $\phi:\C\to\C$ be a $K$-quasiconformal mapping. For every $d>0$ and $K\geq 1$ there exist two positive constants $C=C(K,d)$ and $C_1=C_1(k,d)$ such that
\begin{equation}\label{eqfo2}
\sum_{j\geq0}\frac{\ve_{\mu,a,t}(\phi(B(x,2^jr)))^d}{2^{bj}}\leq C(K,d)\,\ve_{\mu ,a,t}(\phi(B(x,r)))^d
\end{equation}
whenever $0<a<C_1b$. In particular, if $b=2-t$ and $a$ is chosen small enough, the function $\ve$ defined by $\ve(B) = \ve_{\mu,a,t}(\phi(B))^d$ for 
any ball $B$, belongs to $\cG_1\cap\cG_2(t)$.  
\end{lemma}

\begin{proof}
Set $d_j={\rm \diam}(\phi(B(x,2^jr)))$. We have
\begin{align*}
S & =\sum_{j\geq0}\frac{\ve_{\mu,a,t}(\phi(B(x,2^jr)))^d}{2^{bj}} \simeq \sum_{j\geq0}\frac{\ve_{\mu ,a,t}(B(\phi
(x),d_j))^d}{2^{bj}}\\
&  \lesssim \sum_{k\geq0}
\sum_{j:d_0 2^{k}\leq d_j< d_02^{k+1}}\frac{\ve_{\mu ,a,t}(B(\phi(x),2^kd_0))^d}{2^{bj}}.
\end{align*}
For each $j\geq0$ we have
$$\frac{d_j}{d_0} = \prod_{i=1}^j \frac{d_i}{d_{i-1}} = \prod_{i=1}^j \frac{{\rm \diam}(\phi
(B(x,2^ir)))}{{\rm \diam}(\phi(B(x,2^{i-1}r)))}
\leq C(K)^j = 2^{C_2j},$$
with $C_2$ depending on $K$. Thus, for $j,k$ such that $d_0 2^{k}\leq d_j< d_02^{k+1}$,
$$2^j\geq \Bigl(\frac{d_j}{d_0}\Bigr)^{1/C_2}\simeq 2^{k/C_2}.$$
Since the number of $j$s such that $d_0 2^{k}\leq d_j< d_02^{k+1}$ depends only on $K$,
we obtain
$$S\lesssim \sum_{k\geq0}
\sum_{j:d_0 2^{k}\leq d_j< d_02^{k+1}} \frac{\ve_{\mu ,a,t}(B(\phi(x),2^kd_0))^d}{2^{C_1bk}} \leq C 
\sum_{k\geq0} \frac{\ve_{\mu ,a,t}(B(\phi(x),2^kd_0))^d}{2^{C_1bk}},$$
with $C_1=1/C_2$.

First we consider the case $d=1$:
from Lemma \ref{emug2}, if $0<a<C_1b$, we infer that
\begin{align}\label{eqd41}
\sum_{k\geq0} \frac{\ve_{\mu ,a,t}(B(\phi(x),2^kd_0))}{2^{C_1bk}} & = 
\sum_{k\geq0} \frac{1}{2^{(t+C_1b)k}d_0^{\,t}}\int \frac1{\biggl(\dfrac{|\phi(x)-y|}{2^kd_0}\biggr)^{t+a} + 1}\,
d\mu(y) \nonumber\\
& \lesssim \frac1{d_0^{\,t}}\int 
 \frac 1{\biggl(\dfrac{|\phi(x)-y|}{d_0}\biggr)^{t+a} + 1}\,d\mu(y) \nonumber\\ 
 & =
\ve_{\mu ,a,t}(B(\phi(x),d_0))\lesssim \ve_{\mu ,a,t}(\phi(B(x,r))),
\end{align}
and so we are done in this case.

If $d>1$, we set
$$S\lesssim \left( \sum_{k\geq0} \frac{\ve_{\mu ,a,t}(B(\phi(x),2^kd_0))}{2^{\frac{C_1}d bk}} \right)^d
\lesssim  \ve_{\mu ,a,t}(\phi(B(x,r)))^d,$$
by \rf{eqd41}, replacing $C_1$ there by $C_1/d$ (and thus assuming now that $0<a<\frac{C_1}d b$).

When $0<d<1$ we use H\"older inequality, with $p=1/d$:
\begin{align*}
\sum_{k\geq0} \frac{\ve_{\mu ,a,t}(B(\phi(x),2^kd_0))^d}{2^{C_1bk}} & \leq
\biggl(\sum_{k\geq0} \frac{\ve_{\mu ,a,t}(B(\phi(x),2^kd_0))}{2^{C_1bk}}\biggr)^d
\biggl(\sum_{k\geq0} \frac1{2^{C_1bk}}\biggr)^{1/p'}\\
& \leq C \biggl(\sum_{k\geq0} \frac{\ve_{\mu ,a,t}(B(\phi(x),2^kd_0))}{2^{C_1bk}}\biggr)^d.
\end{align*}
If we plug in this inequality the estimate obtained in \rf{eqd41}, then \rf{eqfo2} follows.

The fact that $\ve$ defined by $\ve(B) = \ve_{\mu,a,t}(\phi(B))^d$ for 
any ball $B$ belongs to $\cG_2(t)$ is a consequence of the definition of $\cG_2(t)$ in
\rf{eqsum4} and the estimate \rf{eqfo2}, choosing $b=2-t$ and $0<a<C_1b$.
\end{proof}

\bigskip


\section{Weighted bounds for the Beurling transform}

\subsection{The weight $\omega$ and the Beurling transform}

In this section, $0<t<2$ is fixed. We will prove weighted estimates for the Beurling transform. To describe the class of weights we refer to, let $\cP=\{ P_i\}_{i=1}^N$ be a family of dyadic squares such that $3P_i\cap 3P_j=\varnothing$ if $i\neq j$, and satisfying the $h$-packing condition
\begin{equation}\label{pack}
 \sum_{P\in\cP:P\subset Q} h(P)\leq C_{pack}\, h(Q)\quad \mbox{ for every 
dyadic square $Q$}.
\end{equation}
Here, $h(x,r)=r^t\,\ve(x,r)$ is any gauge function with $\ve\in{\cal G}_2$. Then the weights we are interested in are precisely the following:
\begin{equation}\label{defomega}
\omega = \sum_{P\in\cP} \frac{h(P)}{\ell(P)^2}\,\chi_P.
\end{equation}
These weights already appeared in \cite{Lacey-Sawyer-Uriarte} in the particular case
$\ve(x,r)\equiv1$. It is easy to see that if $\overline P = \cup_{i=1}^N P_i$, then $\omega$ belongs to $A_{1,\overline P}$, the local $A_1$ Muckenhoupt class. That is, for every square $Q\subset \C$,
\begin{equation}\label{eqmax2}
\frac{\omega(Q)}{\ell(Q)^2}\leq C\,\omega(x),\hspace{1cm}\text{ for almost every }x\in\overline P\cap Q.
\end{equation}
 Indeed, let $Q$ be a square containing $x\in
P_i$ such that $\ell(Q)\geq\ell(P_i)$. We have $\omega(Q)\leq C\,h(Q)$ because of the packing condition on the squares from $\mathcal P$, and then
using that $\ve\in\cG_2$ we infer that $\frac{h(x,r)}{r^2}\leq C\,\frac{h(x,s)}{s^2}$ if 
$s\leq r$, and thus
$$\frac{\omega(Q)}{\ell(Q)^2}\leq C\,\frac{h(Q)}{\ell(Q)^2} \leq C\,\frac{h(P_i)}{\ell(P_i)^2} = C\,\omega(x).$$
If $\ell(Q)<\ell(P_i)$, then it is also clear that \rf{eqmax2} holds.
From \rf{eqmax2} we obtain that if $M$ is the classical Hardy-Littlewood maximal operator, then
\begin{equation}\label{eqmax3}
M\omega(x)\leq C\,\omega(x)\hspace{1cm}\text{ for almost every }x\in\overline P.
\end{equation}

Let $M_\omega$ be the centered Hardy-Littlewood maximal function with respect to the $\omega$ measure. That is,
$$M_\omega f(x) = \sup_{r>0} \frac1{\omega(Q(x,r))}\,\int_{Q(x,r)}f(y)\,\omega(y)\,dm(y),$$
where $Q(x,r)$ stands for the square centered at $x$ with side length $2r$, and $m$ denotes the planar Lebesgue measure.
It is well known that $M_\omega$ is of weak type $(1,1)$ and strong type 
$(p,p)$, for $1<p\leq\infty$, with respect to the measure $\omega$.
From the following lemma it turns out that the same is also true for $M$:

\begin{lemma}\label{lemmaxomega}
Let $\omega$ be as above.
There exists some constant $C$ such that
$$M f(x)\leq C\, M_\omega f(x)\qquad\mbox{for $f\in L^1_{loc}(\overline P)$ and $x\in\overline P$.}$$
As a consequence, $M$ is of weak type $(1,1)$ and strong type 
$(p,p)$, $1<p\leq\infty$, with respect to $\omega$.
\end{lemma}

\begin{proof}
Let $f\in L^1_{loc}(\overline P)$
and $Q$ a square containing $x\in P_i$. Consider the minimal square $Q'$ centered at $x$
containing $Q$. Since $\ell(Q')\simeq \ell(Q)$, using \rf{eqmax2} we get
\begin{align*}
\frac1{m(Q)}\int_Q|f|\,dm&\lesssim \frac1{m(Q')}\int_{Q'}|f|\,dm \\
& \lesssim \frac{\inf_{y\in Q'\cap \overline P}\omega(y)}{\omega(Q')}\,\int_{Q'}|f|\,dm\leq
\frac1{\omega(Q')}\int_{Q'}|f|\,\omega\,dm\leq M_\omega f(x).
\end{align*}
Thus $M f(x)\leq C\, Mf_\omega(x)$ and the lemma follows.
\end{proof}


\bigskip
Recall that the Beurling transform of a function $f:\C\to\C$ is given by
$$\cS f(z) = \frac{-1}{\pi}\,\pv \int_\C \frac{f(\xi)}{(z-\xi)^2}\,dm(\xi).$$ The $\epsilon$-truncated Beurling transform is
$$\cS_\epsilon f(z) = \frac{-1}{\pi} \int_{|z-\xi|>\epsilon} \frac{f(\xi)}{(z-\xi)^2}\,dm(\xi),$$
and the maximal Beurling transform, $\cS_* f(z)=\sup_{\epsilon>0}|\cS_\epsilon f(z)|.$
\medskip

\begin{proposition}\label{propo9}
Let $\cP$ be a family of dyadic squares as above, and set
$\overline P = \bigcup_{i=1}^N P_i$. If $\ve\in\cG_2$ and $\omega$ is the 
weight defined by \rf{defomega}, then the Beurling transform, is bounded in $L^p(\omega)$, for $1<p<\infty$,
and of weak type $(1,1)$ with respect to $\omega$. That is,
\begin{equation}\label{strong}
\|\cS(f\chi_{\overline P})\|_{L^p(\omega)}\leq C\,\|f\|_{L^p(\omega)}\quad \mbox{for all $f\in L^p(\omega)$},
\end{equation}
and
\begin{equation}\label{weak**}
\|\cS(f\chi_{\overline P})\|_{L^{1,\infty}(\omega)}\leq C\|f\|_{L^1(\omega)}\quad \mbox{for all $f\in L^1(\omega)$},
\end{equation}
for some $C>0$ depending on $p$ and $C_{pack}$.
\end{proposition}

It is possible to prove the estimate \rf{strong} (which, for $p=2$, is the one needed in connection with quasiconformal distortion) by an appropriate modification of the arguments of 
\cite{Lacey-Sawyer-Uriarte}. 
 However we have preferred to follow a new approach:
 first
we will show the following weak type inequality, which is stronger than \rf{weak**}:
\begin{equation}\label{weak}
\|\cS_*(f\chi_{\overline P})\|_{L^{1,\infty}(\omega)}\leq C\|f\|_{L^1(\omega)}
\quad \mbox{for all $f\in L^1(\omega)$}.
\end{equation}
Then, by means of a \emph{good lambda} inequality, we will deduce that the maximal Beurling transform is bounded in $L^p(\omega)$, for $1<p<\infty$, that is
\begin{equation}\label{eqstrong*}
\|\cS_*(f\chi_{\overline P})\|_{L^p(\omega)}\leq C\,\|f\|_{L^p(\omega)}.
\end{equation}
Clearly, \eqref{strong} follows from this estimate.
We prove \eqref{weak} in the next Lemma.

\begin{lem}\label{weakbstar}
 We have
$$\omega\left(\left\{z\in\C:|\cS_* f(z)|>\lambda\right\}\right)\leq\frac{C}{\lambda}\,\|f\|_{L^1(\omega)}$$
for every $f\in L^1(\omega)$ and $\lambda>0$. 
\end{lem}
\begin{proof}
We have
$$\aligned
&\omega\left\{z\in\C:|\cS_* f(z)|\geq \lambda\right\}=\sum_i\omega\left\{z\in P_i: |\cS_* f(z)|\geq \lambda\right\}\\
&\leq
\sum_i \omega\left\{z\in P_i: |\cS_*(f\chi_{2P_i})(z)|\geq\frac{\lambda}{2}\right\}+
\sum_i \omega\left\{z\in P_i: |\cS_*(f\chi_{\C\setminus 2P_i})(z)|\geq\frac{\lambda}{2}\right\}\\
&=A+B.
\endaligned$$
For $A$ we just use the boundedness of $\cS_*:L^1\to L^{1,\infty}$ with respect to Lebesgue measure,
$$
\aligned
A=\sum_i \omega\Bigl\{z\in P_i: |\cS_*(f\chi_{2P_i})(z)| & \geq\frac{\lambda}{2}\Bigr\}
\\ &
= \sum_i \frac{h(P_i)}{\ell(P_i)^2}\,\Bigl|\Bigl\{z\in P_i: |\cS_*(f\chi_{2P_i})(z)|\geq\frac{\lambda}{2}
\Bigr\}\Bigr|
\\
&\leq 2\|\cS_*\|_{L^1\to L^{1,\infty}}\sum_i \frac{h(P_i)}{\ell(P_i)^2}\, \frac1\lambda
\int |f\chi_{2P_i}|\,dm\\
&= 2\|\cS_*\|_{L^1\to L^{1,\infty}}\, \frac{\|f\|_{L^1(\omega)}}{\lambda}
\endaligned
$$ 
because the squares $3P_i$ are disjoint and $\omega$ coincides with the Lebesgue measure times
$\frac{h(P_i)}{\ell(P_i)^2}$ on every $P_i$. For the remaining term, denoting the center of $P_i$
by $z_i$, one has
$$\aligned
B&=\sum_i \omega\left\{z\in P_i: |\cS_*(f\chi_{\C\setminus 2P_i})(z)|\geq\frac{\lambda}{2}\right\}\\
&\leq\frac2\lambda\sum_i \int_{P_i} |\cS_*(f\chi_{\C\setminus 2P_i})(z)|\,d\omega(z)\\
&\leq\frac2\lambda\sum_i\int_{P_i} \sum_{j\neq i}\int_{P_j}\frac{|f(\xi)|}{|\xi-z|^2}\,dm(\xi)\,d\omega(z)\\
&\leq C\,\frac2\lambda\sum_{j}\left(\sum_{i\neq j}\frac{\omega(P_i)}{|z_j-z_i|^2}\right)\int_{P_j}|f(\xi)|\,dm(\xi).
\endaligned$$
Using the $h$-packing condition \eqref{pack} and the ${\cal G}_2$ condition for $h$, we get
$$
\aligned
\sum_{i\neq j}\frac{\omega(P_i)}{|z_j-z_i|^2}&=\sum_{k=2}^\infty\,\sum_{i: z_i\in 2^kP_j\setminus 2^{k-1}P_j}\frac{h(P_i)}{|z_i-z_j|^2}\leq C\,\sum_{k=2}^\infty\,\sum_{i: P_i\subset 2^{k+1}P_j}\frac{h(P_i)}{(2^k\ell(P_j))^2}\\
&= C\,\sum_{k=2}^\infty \frac{1}{(2^k\ell(P_j))^{2}}\sum_{i: P_i\subset 2^{k+1}P_j}h(P_i)\leq C\,\sum_{k=2}^\infty \frac{h(2^kP_j)}{(2^k\ell(P_j))^{2}}\leq C\,\frac{h(P_j)}{\ell(P_j)^{2}},\\
\endaligned
$$
where we used the fact that if $z_i\in  2^kP_j$, then $P_i\subset 2^{k+1}P_j$ (because 
$3P_i\cap3P_j=\varnothing$). Thus, 
$$B\leq \frac{C}{\lambda}\,\|f\|_{L^1(\omega)}.$$
 The Lemma follows since both $A$, $B$ are bounded by constant multiples of $\frac1\lambda\,\|f\|_{L^1(\omega)}$.
\end{proof}

\bigskip

\begin{proof}[\bf Proof of the boundedness of $\mathcal S_*$ in $L^p(\omega)$, for $1<p<\infty$]
Our goal here is to obtain the following \emph{good lambda} inequality,
\begin{equation}\label{gl}
\omega\bigg(\left\{z: \cS_* f(z)>10\lambda, M_\omega f(z)\leq\gamma\lambda\right\}\bigg)\leq C\,\gamma\,\omega\bigg(\left\{z:\cS_* f(z)>\lambda\right\}\bigg),
\end{equation}
for every $\lambda>0$, and some $\gamma$ small enough. Recall that $M_\omega$ denotes the centered Hardy-Littlewood maximal function with respect to the $\omega$ measure.
By standard arguments, the preceding estimate implies that
$$\|{\mathcal S}_*f\|_{L^p(\omega)}\leq C_p
\|M_\omega f\|_{L^p(\omega)},$$
for $0< p <\infty$. Since $M_\omega$ is bounded in $L^p(\omega)$
for $p>1$, this implies that
$\mathcal S_*$ is bounded in $L^p(\omega)$, $1<p<\infty$. 

\noindent To get \eqref{gl}, let us denote $\Omega_\lambda=\{\cS_* f>\lambda\}$, and let
$$\Omega_\lambda=\bigcup_{j=1}^\infty Q_j$$
be a Whitney decomposition of $\Omega_\lambda$. That is, $\sum_j\chi_{10Q_j}\leq C$, and for every $j$ we have that $100 Q_j\subset \Omega_\lambda$ but $1000 Q_j\nsubseteq\Omega_\lambda$. Let $Q_j$ be a fixed Whitney cube, and assume that there exists $z_j\in Q_j$ such that $M_\omega f(z_j)\leq \gamma\lambda$ (otherwise there is nothing to prove). Let $t_j\in\C\setminus\Omega_\lambda$ the closest point to $Q_j$ in $\C\setminus\Omega_\lambda$. Let $B=B(t_j,c_0\ell(Q_j))$ be a ball centered at $t_j$ and such that $10Q_j\subset B$. We can decompose
$$f=f\,\chi_B+f\,\chi_{\C\setminus B}.$$
For every $z\in Q_j$, the truncated singular integral 
$\cS_\epsilon (f\chi_B)(z)$ can be written as the sum of two terms,
$$\aligned
\cS_\epsilon (f\chi_B)(z)&=\int_{|t-z|\geq \epsilon}\frac{f(t)\,\chi_{B\cap 3B_j}(t)}{(t-z)^2}\,dm(t)+\int_{|t-z|\geq \epsilon}\frac{f(t)\,\chi_{B\setminus 3B_j}(t)}{(t-z)^2}\,dm(t)=I+II,
\endaligned$$
where $B_j=B(z_j,\ell(Q_j))$. For $II$, we use that $|t-z|\simeq \ell(Q_j)$ 
and Lemma \ref{lemmaxomega} to get
\begin{equation}\label{eq1}
II\leq \frac{C}{\ell(Q_j)^2}\int_{B\setminus 3B_j}|f(t)|\,dm(t)\leq C\,M f(z_j)
\leq C\,M_\omega f(z_j)\leq C\,\gamma\lambda,
\end{equation}
and this is uniform in $\epsilon$. Therefore, since $\cS_* f(z)\leq \cS_*(f\chi_B)(z)+\cS_*(f\chi_{\C\setminus B})(z)$, we have 
\begin{equation}\label{A+B}
\aligned
\omega \bigg(\big\{z\in Q_j: &\cS_* f(z)>10\lambda, M_\omega f(z)\leq\gamma\lambda\big\}\bigg)\\
\leq
&\;\omega\bigg(\left\{z\in Q_j: \cS_* (f\chi_B)(z)>2\lambda, M_\omega f(z)\leq\gamma\lambda\right\}\bigg)\\
&+\omega\bigg(\left\{z\in Q_j: \cS_* (f\chi_{\C\setminus B})(z)>8\lambda, 
M_\omega f(z)\leq\gamma\lambda\right\}\bigg)=A+B.
\endaligned
\end{equation}
Choosing $\gamma$ so that $C\gamma<1$ in \eqref{eq1}, we get that
$$
A\leq \omega\bigg(\left\{z\in Q_j: \cS_* (f\chi_{B\cap 3B_j})(z)>\lambda, M_\omega f(z)\leq\gamma\lambda\right\}\bigg).
$$
On the other hand, by Lemma \ref{weakbstar}, 
$$
\omega\left\{\cS_* (f\chi_{B\cap 3B_j})(z)>\lambda\right\}\leq \frac{C}{\lambda}\,\|f\chi_{B\cap 3B_j}\|_{L^1(\omega)}\leq\frac{C\,M_\omega f(z_j)\,\omega(B\cap 3B_j)}{\lambda}.
$$
Therefore
\begin{equation}\label{eqA}
 A\leq C\,\gamma\,\omega(B\cap 3B_j)\leq C\,\gamma\,\omega (7Q_j).
\end{equation}
To estimate $B$, notice that, for $z\in Q_j$,
$$
\aligned
\left|\cS_\epsilon(f\chi_{\C\setminus B})(z)-\cS_\epsilon(f\chi_{\C\setminus B})(t_j)\right|
&\leq\int_{\C\setminus B}|f(t)|\left|\frac{1}{(t-z)^2}-\frac{1}{(t-t_j)^2}\right|\,dm(t)\\
&\leq C\,\ell(Q_j)\int_{\C\setminus B}\frac{|f(t)|\,|t-\frac{z+t_j}{2}|}{|t-z|^2\,|t-t_j|^2}dm(t)\\
&\leq C\,\ell(Q_j)\int_{\C\setminus B}\frac{|f(t)|}{|t-z_j|^3}dm(t)\\
&\leq C\,\ell(Q_j)\int_{\C\setminus 3B_j}\frac{|f(t)|}{|t-z_j|^3}\,dm(t)\\
&\leq C\,Mf(z_j)\leq C\,M_\omega f(z_j).
\endaligned
$$
The next to the last step above follows as usually after decomposing into dyadic annuli. Since $B\supset B(z,3\ell(Q_j))$, to compute $\cS_*(f\chi_{\C\setminus B})(z)=\sup_{\epsilon>0}|\cS_\epsilon(f\chi_{\C\setminus B})(z)|$ it suffices to take $\epsilon>3\ell(Q_j)$. Therefore
$$
\aligned
|\cS_\epsilon(f\chi_{\C\setminus B})(t_j)-\cS_\epsilon(f)(t_j)|&=\left|\int_{|t-t_j|>\epsilon}\frac{f(t)\,\chi_B(t)}{(t-t_j)^2}dm(t)\right|\\
&\leq\frac{C}{\ell(Q_j)^2}\int_{|t-t_j|>\epsilon}|f(t)|\,\chi_B(t)dm(t)\\
&=\frac{C}{\ell(Q_j)^2}\int_{\epsilon<|t-t_j|<c_0\ell(Q_j)}|f(t)|\,dm(t)\\
&\leq\frac{C}{\ell(Q_j)^2}\int_{|t-z_j|\leq 2c_0\ell(Q_j)}|f(t)|dm(t)\leq C\,Mf(z_j).
\endaligned
$$
Summarizing, we get
$$
\aligned
|\cS_\epsilon&(f\chi_{\C\setminus B})(z)|\\
&\leq |\cS_\epsilon(f\chi_{\C\setminus B})(z)-\cS_\epsilon(f\chi_{\C\setminus B})(t_j)|+|\cS_\epsilon(f\chi_{\C\setminus B})(t_j)-\cS_\epsilon f(t_j)|+|\cS_\epsilon f(t_j)|\\
&\leq C M_\omega f(z_j)+|\cS_\epsilon f(t_j)|.
\endaligned
$$
Therefore, if $z$ belongs to $\left\{z: \cS_* (f\chi_{\C\setminus B})(z)>8\lambda,  M_\omega f(z)\leq\gamma\lambda\right\}$, we get that
$$
\aligned
8\lambda<|\cS_*(f\chi_{\C\setminus B})(z)|\leq C M_\omega f(z_j)+|\cS_* f(t_j)|\leq C\gamma\lambda+\lambda.
\endaligned
$$
because $t_j\notin\Omega_\lambda$. In particular, back to \eqref{A+B}, for small enough $\gamma$ we must have $B=0$. Now, with the help of \eqref{eqA} we get
$$
\omega\bigg(\left\{z\in Q_j: \cS_* f(z)>10\lambda, M_\omega f(z)\leq\gamma\lambda\right\}\bigg)\leq C\,\gamma\,\omega(7Q_j).
$$
Since the squares $7Q_j$ have bounded overlap, summing in $j$ we obtain
$$\aligned
\omega\bigg(\left\{z: \cS_* f(z)>10\lambda, M_\omega f(z)\leq\gamma\lambda\right\}\bigg)
\leq C\,\gamma\,\omega\left(\left\{z:\cS_* f(z)>\lambda\right\}\right),
\endaligned$$
which is \eqref{gl}. 
\end{proof}

\bigskip


\section{Quasiconformal distortion of $h$-contents}

\subsection{Conformal outside}
Let $\phi$ be a $K$-quasiconformal mapping on $\C$, and let $\ve_0\in\cG_1$. Consider the associated gauge function $h_0(x,r) = r^t\,\ve_0(x,r)$, for a fixed $0<t<2$.\\
\\
For every $x\in\C$ and $r>0$, denote
$$\ve(x,r) = \ve_0(\phi(B(x,r)),\qquad
h(x,r) = r^t\,\ve(x,r),$$ and suppose that $\ve\in\cG_2$.
This fact is crucial in this subsection.\\
\\
Let now ${\cal P}=\{P_i\}$ be a finite family of dyadic squares, with disjoint triples, and satisfying the packing condition 
\begin{equation}\label{packingconstant}
\sum_{P\in\cP:P\subset Q} h(P)\leq C_{pack}\, h(Q)\quad \mbox{ for every dyadic square $Q$}.
\end{equation}
If we introduce the weight
$$\omega=\sum_{P\in{\cal P}}\frac{h(P)}{|P|}\,\chi_P\simeq\sum_{P\subset\cP}\frac{\ve(P)}{\ell(P)^{2-t}}\,\chi_P$$
then it follows from Proposition \ref{propo9} that the Beurling transform is bounded in $L^2(\omega)$,
$$\|{\cal S}(f\chi_{\overline P})\|_{L^2(\omega)}\leq \|\cal S\|_{L^2(\omega)}\,\|f\|_{L^2(\omega)}$$
and the norm $\|{\cal S}\|_{L^2(\omega)}$ depends only on the packing constant $C_{pack}$. Therefore there exists a number $\delta>0$, depending only on $C_{pack}$, such that
\begin{equation}\label{Kpetit}
1\leq K<1+\delta\hspace{1cm}\Rightarrow \hspace{1cm}\frac{K-1}{K+1}\,\|{\cal S}\|_{L^2(\omega)}<1.
\end{equation}

In establishing quasiconformal distortion estimates for the $h$-contents we need to 
normalize the mappings. A very convenient normalization is obtained by requiring that $\phi$ is
conformal outside the unit disk and has the development
$$\phi(z) = z + {\mathcal O}(1/z),\qquad |z| > 1.$$
We call such $\phi$ as {\it principal quasiconformal mapping}.

\begin{lemma} \label{lemkpetit} 
Let $\phi,\ve_0,\ve,\cP,\omega$ be as above and suppose moreover that the quasiconformal mapping $\phi$ is principal and conformal
outside $\overline P=\bigcup_{P\in\cP} P$. Then we have
$$\sum_{P\in\cP}h_0(\phi(P))\leq C\,\sum_{P\in\cP} h(P),$$
for some constant $C=C(K)>0$.
\end{lemma}
\begin{proof}
First of all, by the definition of $\ve_0$ on arbitary sets, we see that
$$
h_0(\phi(P_i))=\inf_{\phi(P_i)\subset B} h_0(B)=\inf_{\phi(P_i)\subset B(x,r)} r^t\,\ve_0(x,r)\simeq \diam(\phi(P_i))^t\,\ve(P_i),
$$
with constants that only depend on $K$. Thus, by H\"older's inequality we get
\begin{align*}
\sum_i  h_0(\phi(P_i))
&\leq C_K\,\sum_i \diam(\phi(P_i))^t\ve(P_i)\\
&\leq C_K\Bigl(\sum_i \diam(\phi(P_i))^2\ell(P_i)^{t-2}\ve(P_i)\Bigl)^{t/2}\Bigl(\sum_i\ell(P_i)^t\ve(P_i)\Bigr)^{\frac{2-t}2}\\
& = C_K\biggl(\sum_i \diam(\phi(P_i))^2\,\frac{\omega(P_i)}{\ell(P_i)^2}\biggl)^{t/2}\Bigl(\sum_i \omega(P_i)\Bigr)^{\frac{2-t}2}
=: C_K \,A^{t/2}\, \omega(\overline P)^{\frac{2-t}2}.
\end{align*}
To estimate $A$, we start by getting from quasisymmetry that
\begin{align*}
A & \leq C_K \sum_i \int_{P_i} J(z,\phi)\,\omega(z)\,dm(z) 
\end{align*}
Now, as $\phi$ is a principal quasiconformal mapping, $\phi(z)-z={\cal C}g(z)$, where ${\cal C}g$ is the Cauchy transform of the Neumann series
$$g=\sum_{n=0}^\infty(\nu\S)^n(\nu),$$
and $\nu(z)=\frac{\overline\partial\phi(z)}{\partial\phi(z)}$ whenever $\partial\phi(z)\neq 0$ (otherwise we simply set $\nu(z)=0$). 
Since $K<1+\delta$, \eqref{Kpetit} says that the above series converges absolutely in $L^2(\omega)$
 (using the key fact that $\supp(\nu)\subset \overline P$ by the conformality of $\phi$ off $\overline{P}$), and moreover one easily gets
 $$\|g\|_{L^2(\omega)}\leq 
 \frac{\|\nu\|_\infty}{1-\|\nu\|_\infty\,\|{\S}\|_{L^2(\omega)}}\,\omega({\overline{P}})^{1/2}
 .$$
Then arguing as in \cite{Lacey-Sawyer-Uriarte} (see also \cite{astalaareadistortion}) 
$$\aligned
\sum_i\int_{P_i} J(z,\phi)\,\omega(z)\,&dm(z)=\int_{\overline{P}}
\Bigl(|\partial \phi(z)|^2-|\bar\partial\phi(z)|^2\Bigr)\,\omega(z)\,dm(z)
\\
&= \int_{\overline{P}}\left( 1+|{\S}g(z)|^2+2\Re({\S}g(z))-|g(z)|^2\right)\,\omega(z)\,dm(z)\\
&\leq \left(\omega({\overline{P}}) + \|{\S}\|_{L^2(\omega)}^2\,\|g\|_{L^2(\omega)}^2+2 \omega({\overline{P}})^\frac12\,\|{\S}\|_{L^2(\omega)}\,\|g\|_{L^2(\omega)}\right)\\
&\leq \omega({\overline{P}})\left( 1 + \,\frac{\|{\S}\|_{L^2(\omega)}^2\,\|\nu\|_\infty^2}{\left(1-\|\nu\|_\infty\,\|{\S}\|_{L^2(\omega)}\right)^2}+\frac{2\,\|{\S}\|_{L^2(\omega)}\,\|\nu\|_\infty}{1-\|\nu\|_\infty\,\|{\S}\|_{L^2(\omega)}}\right)\\
&\leq C(K)\,\omega({\overline{P}}),
\endaligned
$$
by \rf{Kpetit}. Thus,
$$\sum_{P\in\cP}h_0(\phi(P))\leq C(K)\, \omega(\overline P)=C(K)\,\sum_{P\in\cP}h(P),$$
and so the lemma follows.
\end{proof}


\subsection{Conformal inside}

We will prove now an h-version of \cite[Theorem 2.2]{ACMOU}. Here
the point is to use quasiconformal mappings that are conformal inside a finite disjoint union
of quasidisks, allowing improved integrability for the gradient \cite{astalanesi}. Let us
emphasize that no $\cG_2$ assumption will be needed here.

\begin{theo}\label{nesi}
Let $\phi:\C\to\C$ be $K$-quasiconformal, principal, and conformal outside $\D$. Assume that $Q_i\subset\D$ are pairwise disjoint $K$-quasidisks, and that $\phi$ is conformal in $\Omega=\cup_iQ_i$. For a fixed $\varepsilon_0\in{\cal G}_1$ and $0<t<2$, let $t'=\frac{2Kt}{2+(K-1)t}$ and $h_0(r)=r^{t'}\ve_0(x,r)$. Set
$$\ve(B)=\ve_0(\phi(B))^\frac{Kt}{t'}\hspace{1cm}\text{and}\hspace{1cm} h(r)=r^{t}\,\ve(x,r).$$ 
Then
$$
\sum_i h_0(\phi(Q_i))\leq C(K,t)\,\left(\sum_i h(Q_i)\right)^\frac{t'}{Kt}.
$$
\end{theo}

\begin{proof}
From the quasisymmetry of $\phi$, and the doubling properties of $\ve_0$ and $\ve$,
$$
h_0(\phi(Q_i))
=\inf_{\phi(Q_i)\subset B} h_0(B)
=\inf_{\phi(Q_i)\subset B(x,r)}r^{t'} \ve_0(x,r)
\simeq\diam(\phi(Q_i))^{t'}\,\ve(Q_i)^\frac{t'}{Kt}$$
with constants that depend only on $K$. By quasisymmetry again and H\"older's inequality, we get
$$\aligned
\diam(\phi(Q_i))
&\leq C_K\,\left(\int_{Q_i}J(z,\phi)\,dm(z)\right)^\frac12\\
&\leq C_K\left(\int_{Q_i}J(z,\phi)^\frac{K}{K-1}\,dm(z)\right)^\frac{K-1}{2K}\,\diam(Q_i)^\frac1K
\endaligned$$
since quasidisks have diameter comparable to the square root of its area. Therefore, by H\"older's inequality and the improved borderline integrability of quasiconformal mappings (\cite{astalanesi}),
$$\aligned
\sum_i h_0(\phi(Q_i)) 
&\leq C(K)\,\sum_i \diam(\phi(Q_i))^{t'}\,\ve(Q_i)^\frac{t'}{Kt}\\
&\leq C(K)\sum_i\left(\int_{Q_i}J(z,\phi)^\frac{K}{K-1}\,dm(z)\right)^{\frac{K-1}{2K}\,t'}\,\diam(Q_i)^\frac{t'}{K}\,\ve(Q_i)^\frac{t'}{Kt}\\
&\leq C(K)\left(\sum_i\int_{Q_i}J(z,\phi)^\frac{K}{K-1}\,dm(z)\right)^{\frac{K-1}{2K}\,t'}\,
\left(\sum_i\diam(Q_i)^t\,\ve(Q_i)\right)^\frac{t'}{Kt}\\
&\leq C(K)\, \left(\sum_i h(Q_i)\right)^\frac{t'}{Kt}\\\endaligned
$$
as claimed.
\end{proof}

\subsection{The main lemma on distortion of $h$-contents. }We are now ready to prove the main estimate on the distortion of $h$-contents by quasiconformal mappings with small distortion.

\begin{lemma}\label{mainlem}
Let $E\subset B(0,1/2)$ be compact and $\phi:\C\to\C$ a principal $K$-quasiconformal mapping, conformal on $\C\setminus \bar \D$. Let $\ve\in\cG_1$, and assume that $(\ve\circ\phi)^\frac{Kt}{t'}\in \cG_2(t)$. Set 
$$h(x,r)=r^{t'}\ve(x,r)$$
and
\begin{equation*}
\wt \ve(x,r)= \ve(\phi(B(x,r)))^\frac{Kt}{t'},\qquad \wt h(x,r)= r^{t}\wt\ve(x,r).
\end{equation*}
If $\delta$ is as in \eqref{Kpetit} and $K<1+\delta$, then
$$ \M^{h}(\phi(E)) \leq C(K) \, \M^{\wt h}(E)^\frac{t'}{Kt}.$$
\end{lemma}
\begin{proof}
Let us fix $\eta>0$. As in \cite{Lacey-Sawyer-Uriarte} we find a finite family ${\cal P}=\{P_1,\dots,P_N\}$ of dyadic cubes, with disjoint triples, such that $E\subset\cup_i 12P_i$ and
$$\sum_i \wt h(P_i)\leq C( \M^{\wt h}(E)+\eta).$$
Further, we may also assume that the packing condition \eqref{packingconstant} is satisfied, for instance, with constant $C_{pack}=1$, so that $\delta>0$ in \eqref{Kpetit} is fixed. We now decompose $\phi=\phi_2\circ\phi_1$, where both $\phi_1$, $\phi_2$ are principal $K$-quasiconformal
mappings. Moreover, we require $\phi_1$ to be conformal in $\C\setminus\cup_iP_i$, and $\phi_2$ to be conformal on $\cup_i\phi_1(P_i)$ . \\
\\
Now, we note that $\phi(E)$ can be covered by the quasidisks $\phi(12P_i)$. We can then estimate the $h$-content of $\phi(E)$ with the help of the quasisymmetry,
$$\aligned
 \M^{h}(\phi(E))&\leq \sum_i h(\phi(12P_i))\leq C(K)\,\sum_i h(\phi(P_i))\\
&= C(K)\,\sum_ih(\phi_2\circ\phi_1(P_i)).
\endaligned$$
Since $\phi_1$ is a global $K$-quasiconformal mapping, each $\phi_1(P_i)$ is a $K$-quasidisk. Hence, by Lemma \ref{nesi}, if we define the new gauge function $ h_0(x,r)=r^{t}\,\ve_0(x,r)$, with
$\ve_0(B)=\ve(\phi_2(B))^\frac{Kt}{t'}$, then we have
$$
\sum_ih(\phi_2\circ\phi_1(P_i))\leq C(K)\,\left(\sum_i  h_0(\phi_1(P_i))\right)^\frac{t'}{Kt}.
$$
To estimate the sum on the right hand side above, we use Lemma \ref{lemkpetit}. Indeed, the composition $\ve_0\circ\phi_1=(\ve\circ\phi)^\frac{Kt}{t'}$ certainly belongs to $\cG_1$, and by assumption it also belongs to $\cG_2(t)$. Hence Lemma \ref{lemkpetit} gives us another gauge $h_1(x,r)=r^t\,\ve_1(x,r)$, with
$\ve_1(D)=\ve_0(\phi_1(D))$,
 such that  
$$
\sum_i h_0(\phi_1(P_i))\leq C(K)\,\sum_i h_1(P_i).
$$
But then
$$
\ve_1(D)=\ve_0(\phi_1(D))\simeq\ve(\phi_2\circ\phi_1(D))^\frac{Kt}{t'}=\ve(\phi(D))^\frac{Kt}{t'}=\wt\ve(D),
$$
so that $h_1(P_i)\simeq \wt h(P_i)$. Summarizing,
$$\aligned
 \M^{h}(\phi(E)) &\leq C(K)\,\left(\sum_i  h_0(\phi_1(P_i))\right)^\frac{t'}{Kt}
\leq C(K)\,\left(\sum_i h_1(P_i)\right)^\frac{t'}{Kt}
\\
&\leq C(K)\,\left(\sum_i \wt h(P_i)\right)^\frac{t'}{Kt}\leq C(K)\,\left( \M^{\wt h}(E)+\eta\right)^\frac{t'}{Kt}.
\endaligned$$
Now letting $\eta\to 0$, the claim follows.
\end{proof}

\noindent
Our next goal is to remove the smallness assumption $K<1+\delta$ in the previous Lemma. 
This is done by means of a standard factorization argument.

\begin{lemma}\label{lemkgran}
Let $0<t<2$. Let $\ve\in\cG_1$ and set $h(x,r)=r^{t'}\ve(x,r)$. Suppose that for any principal $K$-quasiconformal mapping $\psi:\C\to\C$ the function $(\ve\circ\psi)^d$ belongs to $\cG_2$ for any $\frac{t'}{Kt}\leq d\leq 1$. Let $E\subset B(0,1/2)$ be compact and $\phi:\C\to\C$ a principal $K$-quasiconformal mapping, conformal on $\C\setminus \bar \D$. Set 
\begin{equation*}
\wt\ve(x,r)= \ve(\phi(B(x,r)))^{\frac{Kt}{t'}},\qquad \wt h(x,r)= r^{t}\wt\ve(x,r).
\end{equation*}
Then we have
$$ \M^{h}(\phi(E)) \leq C(K) \, \M^{\wt h}(E)^\frac{t'}{Kt}.$$
\end{lemma}

\begin{proof}
We factorize $\phi$ so that $\phi = \phi_n\circ\cdots \phi_1$, where each $\phi_i$ is a $K^{1/n}$-quasiconformal mapping, conformal on $\C\setminus \phi_{i-1}\circ\dots\circ\phi_1(\D)$. We can further do this so that
$$K^{1/n}<1+\delta$$
for $\delta$ as in \eqref{Kpetit}. Of course, such $n$ will depend on $K$ and also on the packing constant $C_{pack}$ in \eqref{packingconstant}. So we have
$$E = E_0 \stackrel{\phi_1}{\longrightarrow} E_1
\stackrel{\phi_2}{\longrightarrow} \ldots \stackrel{\phi_{n-1}}{\longrightarrow} E_{n-1} \stackrel{\phi_n}{\longrightarrow} E_n
=\phi(E).$$
We now denote $t_0=t$, and for $0\leq j\leq n-1$, we take $t_{j}$ so that
$$\frac{1}{t_{j+1}}-\frac12=\frac{1}{K^\frac1n}\left(\frac{1}{t_j}-\frac12\right).$$
In particular, $t_n=t'$. For any ball $B$, we also set $\ve_n(B)=\ve(B)$, $h_n(B)=h(B)$, and for $j=n-1,n-2,\dots,1,0$ let 
$$\ve_j(B) = \ve_{j+1}(\phi_{j+1}(B))^{\frac{K^\frac1n t_j}{t_{j+1}}} 
= \ve_j(\phi_j(B))^{\frac{t_j}{K^{1/n}\,t_{j-1} }}$$
and $$h_j(x,r)= r^{t_j}\,\ve_j(x,r).$$
Note that therefore $$\ve_0(B)=\ve_1(\phi_1(B))^\frac{K^{1/n}t_0}{t_1}=\ve_2(\phi_2\circ\phi_1(B))^\frac{K^{2/n}t_0}{t_2}=\ve_n(\phi(B))^\frac{Kt}{t'}=\wt\ve(B).$$ 
By recursively using Lemma \ref{mainlem}, we have
\begin{align*}
 \M^{h}(\phi(E)) & =  \M^{h_n}(E_n)\leq C(K)\, \M^{h_{n-1}}(E_{n-1})^{\frac{t_n}{K^{1/n} t_{n-1}}} \\
&\leq C(K)\,C(K)^\frac{t_n}{K^{1/n}t_{n-1}}\, \M^{h_{n-2}}(E_{n-2})^{\frac{t_n}{K^{2/n} t_{n-2}}}\\
&\leq  C(K)^{1+\frac{t_n}{K^{1/n}t_{n-1}}+\dots+\frac{t_n}{K^{(n-1)/n}t_{1}}}\, \M^{h_0}(E_0)^{\frac{t_n}{K\,t_0 }}\\
&=C(K)^{1+(n-1)\left(1-\frac{t'}{2}\right)+\frac{t'}{2}\frac{K-K^{1/n}}{K^{1/n}-1}}\, \M^{h_0}(E)^{\frac{t'}{K\,t }}.
\end{align*}
Now, since $C_{pack}$ is fixed, we see that the constant above depends only on $K$. Therefore we can rewrite this in the following way
$$
 \M^{h}(\phi(E))\leq C(K)\, \M^{\wt h}(E)^\frac{t'}{Kt}
$$
as claimed.
\end{proof}

\bigskip


\section{Quasiconformal distortion of Riesz capacities.}

The following lemma describes the relationship between Riesz capacities and $h$-contents.

\begin{lemma}\label{capMh}
Let $1<p<\infty$ and $0<\alpha p<2$, and let $E\subset\C$ be compact. Then
\begin{equation}\label{supr}
\dot{{\cal C}}_{\alpha,p}(E)\simeq \sup\left\{ \M^h(E)\right\}
\end{equation}
where the supremum on the right hand side runs over all gauge functions 
$$h(x,r)=r^{2-\alpha p}\,\ve(x,r)$$ 
with $\ve\in{\cal G}_1$ and such that
\begin{equation}\label{normaliz}
\int_0^\infty\frac{\ve(x,r)^{p'-1}}{r}\,dr\leq 1,\qquad x\in\C.
\end{equation}
\end{lemma}

\begin{proof}
We will use the characterization of $\dot{\cal C}_{\alpha,p}$ in terms of Wolff potentials (see Subsection \ref{strategy}). Let us consider a measure $\mu$ supported on $E$ and such that $\dot{W}_{\alpha,p}^\mu(x)\leq 1$ for all $x\in\C$, and let $t=2-\alpha p$, so that $0<t<2$. For small enough $a>0$, construct $h_{\mu,a,t}$ as we did in \eqref{defhx}. Recall that $h_{\mu,a,t}$ belongs to the class ${\cal G}_1$. By Lemma \ref{Mmu}, $ \M^{h_{\mu,a,t}}(E)\geq C\,\mu(E)$. Decomposing the integrals into annuli, for all $x\in\C$ we get 
\begin{align*}
\int_0^\infty \ve_{\mu,a,t}(x,r)^{p'-1}\frac{dr}r &= \int_0^\infty\frac1{r^{t(p'-1)}} \biggl(\int \psi_{a,t}\Bigl(\frac{y-x}r\Bigr)d\mu(y)\biggr)^{p'-1}\frac{dr}r\\
& \leq C \sum_{j\in \Z} 2^{-t(p'-1)j} \Bigl(\sum_{k>j}\mu(B(x,2^k)) 2^{(t+a)(j-k)}\Bigr)^{p'-1}\\
& \leq C \sum_{j\in \Z} 2^{-t(p'-1)j} \sum_{k>j}\mu(B(x,2^k))^{p'-1} 2^{(p'-1)(t+\frac a2)(j-k)},
\end{align*}
where we applied H\"older's inequality for $p'-1>1$, and the fact that $(c+d)^{p'-1}\leq c^{p'-1}+ d^{p'-1}$ otherwise. Thus,
\begin{align}\label{ed**}
\int _0^\infty\ve_{\mu,a,t}(x,r)^{p'-1}\frac{dr}r &
 \lesssim
\sum_{k\in \Z}  \mu(B(x,2^k))^{p'-1}\, 
 2^{-(p'-1)(t+\frac a2)k}
 \sum_{j<k}  2^{t(p'-1)\frac a2 j} \nonumber \\
 & \simeq \sum_{k\in \Z}  \mu(B(x,2^k))^{p'-1}\, 
 2^{-(p'-1)tk}
 \lesssim  \dot W^\mu_{\alpha,p}(x)\lesssim 1.
\end{align}
Above we allow constants in the estimates to depend on $\alpha,p,t,a$, but not on $\mu$. Therefore, if $\mu$ is admissible for $\dot{\cal C}_{\alpha,p}(E)$ then $h_{\mu,a,t}$ is admissible for the supremum in \eqref{supr} and
$$\sup_h\left\{ \M^h(E)\right\}\geq C\, \M^{h_{\mu,a,t}}(E)\geq C\,\mu(E).$$
Conversely, let us fix any gauge function $h(x,r)=r^{2-\alpha p}\,\ve(x,r)$ in ${\cal G}_1$ satisfying \eqref{normaliz} and such that $ \M^h(E)>0$. By Lemma \ref{frostman}, there exists a measure $\mu$ supported on $E$ such that
$$\mu(B)\leq C\,h(B),$$
and furthermore, we can choose $\mu$ so that $\mu(E)\geq C\, \M^h(E)$. But then
$$
\dot W_{\alpha,p}^\mu(x)=\int_0^\infty\left(\frac{\mu(B(x,r))}{r^{2-\alpha p}}\right)^{p'-1}\frac{dr}{r}\leq\int_0^\infty\frac{\ve(x,r)^{p'-1}}{r}\,dr\leq 1.
$$
Therefore $\mu$ is admissible for $\dot{{\cal C}}_{\alpha,p}(E)$, and $\dot{{\cal C}}_{\alpha,p}(E)\geq C\,\mu(E)\geq C\, \M^h(E).$
\end{proof}

\noindent
Let us remark that we can further restrict the class of admissible functions $h$ in the above supremum. In fact, it follows from the proof above that
\begin{equation}\label{caphmu}
\dot{{\cal C}}_{\alpha,p}(E)\simeq \sup\left\{ \M^{h_{\mu,a,t}}(E): \supp(\mu)\subset E, 
\int_0^\infty\left(\frac{h_{\mu,a,t}(x,r)}{r^t}\right)^{p'-1}\,\frac{dr}{r}\leq 1\right\},
\end{equation}
where $t=2-\alpha p$, as above. On the other hand, we emphasize the fact that Lemma \ref{capMh} does not hold if we restrict the supremum to gauge functions $h$ invariant under translations (see \cite[Remark 5.6.4]{adamshedberg}).
Finally, if condition \eqref{normaliz} is replaced by
$$\lim_{r\to 0}\frac{h(x,r)}{r^{2-\alpha p}}= \lim_{r\to 0}\ve(x,r)=0\hspace{1cm}\text{ uniformly in }x$$
then we obtain the lower $t$-dimensional Hausdorff content, which vanishes exactly on sets having $\sigma$-finite $t$-dimensional Hausdorff measure $ \H^t$ (see \cite{sionsjerve} for more details). \\
\\
\noindent
Before proving Theorem \ref{teocap2}, we need the following auxiliary result.

\begin{lemma} \label{lemcg}
Let  $\phi:\C\rightarrow\C$ be a $K$-quasiconformal mapping, and 
$\ve_0\in\cG_1$. 
Define $\ve_1(B) = \ve_0(\phi(B))$ for any ball $B\subset\C$. 
For any $s>0$ we have
$$\int_0^\infty \ve_1(x,r)^s\,\frac{dr}r\leq C(K,s)\int_0^\infty \ve_0(\phi(x),r)^s\,\frac{dr}r.
$$
\end{lemma}

\begin{proof}
We have
$$
\int_0^\infty 
\ve_0(\phi(B(x,r)))^s\, \frac{dr}r 
\leq C(s)\sum_{j\in\Z} \ve_0(\phi(B(x, 2^j)))^s.$$
Denote now $r_j=\diam(\phi(B(x, 2^j))$. We obtain
\begin{align*}
\sum_{j\in\Z} \ve_0(\phi(B(x, 2^j))^s & = \sum_{k\in\Z} \sum_{j:2^k\leq r_j< 2^{k+1}}\!\!\!\!\ve_0(\phi(B(x, 2^j))^s \\ &\lesssim \sum_{k\in\Z} \sum_{j:2^k\leq r_j< 2^{k+1}}\!\!\!\!\ve_0(B(\phi(x), r_j))^s \\
& \lesssim C(K) \sum_{k\in\Z} \ve_0(B(\phi(x), 2^k))^s\leq C(K,s)\int_0^\infty 
\ve_0(\phi(x),r)^s\, \frac{dr}r,
\end{align*}
where we took into account that $\#\{j:2^k\leq r_j< 2^{k+1}\}\leq C(K)$. This follows easily from the fact that  the moduli of the annuli $B(x,2^{j+1})\setminus B(x,2^j)$ is
$K$-quasi-invariant by $\phi$.
\end{proof}

\medskip
\begin{proof}[\bf Proof of Theorem \ref{teocap2}] By standard methods, we may assume that $\phi$ is a principal quasiconformal mapping, conformal on $\C\setminus \bar \D$, and that $E\subset B(0,1/2)=\frac12 B$ (and so $\diam(\phi(B))\simeq 1$). We can further assume that $\dot{\cal C}_{\beta,q}(\phi(E))>0$, since otherwise the statement is obvious. Let $a>0$ be small enough, and $t'=2-\beta q$. By \eqref{caphmu}, we can find a finite Borel measure $\mu$ supported on $\phi(E)$ such that  
$$\dot{\cal C}_{\beta,q}(\phi(E))\simeq\M^{h_{\mu,a,t'}}(\phi(E)),$$
and
$$\int_0^\infty\left(\frac{h_{\mu,a,t'}(x,r)}{r^{t'}}\right)^{q'-1}\frac{dr}{r}\leq 1$$ 
for all $x\in\C$. Writing $h_{\mu,a,t'}(x,r)=r^{t'}\,\ve_{\mu,a,t'}(x,r)$, we proved in Lemma \ref{lemkgran} that if
$$\wt h(x,r)=r^t\,\wt\ve(x,r)\hspace{1cm}\text{and}\hspace{1cm}\wt\ve(B)=\ve_{\mu,a,t'}(\phi(B))^\frac{Kt}{t'}$$
then we have the inequality
\begin{equation}\label{ineq}
\M^{h_{\mu,a,t'}}(\phi(E)) \leq C(K)  \M^{\wt h}(E)^\frac{t'}{Kt},
\end{equation}
with a constant $C(K)>0$ depending only on $K$. Furthermore, using our choice
$$
\frac{1}{p'-1}=p-1=\frac{Kt}{t'}\,(q-1)=\frac{Kt}{t'}\,\frac{1}{q'-1},
$$
together with Lemma \ref{lemcg}, we get that
\begin{align*}
\int_0^\infty\left(\frac{\wt{h}(x,r)}{r^{2-\alpha p}}\right)^{p'-1}\,\frac{dr}{r}
&=\int_0^\infty \wt\ve(x,r)^{p'-1}\,\frac{dr}r=\int_0^\infty\ve_{\mu,a,t'}(\phi(B(x,r)))^{\frac{Kt(p'-1)}{t'}} \,\frac{dr}r \\
& \leq C\int_0^\infty\ve_{\mu,a,t'}(\phi(x),r)^{\frac{Kt(p'-1)}{t'}} \,\frac{dr}r\\
&= C\, \int_0^\infty\ve_{\mu,a,t'}(\phi(x),r)^{q'-1} \,\frac{dr}r\leq C
\end{align*}
for all $x\in\C$. Therefore, by Lemma \ref{capMh}, $\wt h$ is admissible for $\dot{\cal C}_{\alpha,p}(E)$, and taking supremum in \eqref{ineq} we get that
$$\dot{\cal C}_{\beta,q}(\phi(E))\leq C\,\dot{\cal C}_{\alpha,p}(E)^\frac{t'}{Kt},$$
as desired.
\end{proof}

\bigskip

\section{Proof of Corollary \ref{teosigma} using Lemma \ref{lemkgran}}

Although Corollary \ref{teosigma} is an immediate consequence of Theorem \ref{teopri},
which will be proved in next section, we would like to show that it also follows rather easily from 
Lemma \ref{lemkgran}.
\\
\\
Suppose that $\H^{t'}(\phi(E))$ is non $\sigma$-finite. Then it supports some non zero measure $\mu$ such that
$\theta_\mu^{t'}(x)=0$ at $\mu$-a.e. $x\in\phi(E)$. We can assume $\theta_\mu^{t'}(x)=0$ for all $x\in\C$,
replacing $\mu$ by its restriction to some nice subset if necessary . It is easy to check that
this implies that $\ve_{\mu,a,t'}(x,r)\to 0$ as $r\to 0$, for all $x\in\C$ (one only has to write $\ve_{\mu,a,t'}(x,r)$
as a convex combination of $\mu(B(x,s))/s^{t'}$, $s\in(0,\infty)$). 
As in Theorem \ref{teocap2}, using Lemma \ref{lemkgran} and Frostman Lemma, we deduce that
there exists another non zero measure $\nu$ supported $E$ such that $\nu(x,r) \leq  r^{t}\wt \ve(x,r)$ 
for all $x\in\C$ and $r>0$, with 
$$\wt\ve(x,r)= \ve_{\mu,a,t'}(\phi(B(x,r)))^{\frac{Kt}{t'}}.$$
Since $\phi$ is continuous, we have
$$\lim_{r\to0}\wt\ve(x,r)= \lim_{r\to0}\ve_{\mu,a,t'}(\phi(B(x,r)))^{\frac{Kt}{t'}}
=0 \quad\mbox{ for all $x\in\C$,}$$
which implies that $\theta_\nu^{t}(x)=0$ for all $x\in\C$. The fact that $E$ supports a non zero measure with zero $t$-density $\nu$-a.e. implies that $\H^{t}(E)$ is non $\sigma$-finite.
\fiproof

\bigskip

\section{Quasiconformal distortion of Hausdorff measures.}\label{tolsapreprint}
\noindent

First we need the following technical lemma.

\begin{lemma}\label{lemdens}
Let $0<s\leq 2$.
Let $\mu$ be a finite Borel measure, and let $x\in\C$ and $\theta_1>0$ be such that
\begin{equation}\label{eq21}
\frac{\mu(B(x,r))}{r^s} \leq \theta_1 \quad\mbox{ if $0< r\leq \delta$.}
\end{equation}
Then there exists $\delta'>0$, depending only on $\delta,a,s, \theta_1$ and $\mu(\C)$, such that
$$\ve_{\mu,a,s}(x,r)\leq \theta_2 \quad\mbox{ if $0< r\leq \delta'$,}
$$
with $\theta_2=C\,\theta_1$, with $C$ depending only on $a,s$.
\end{lemma}

\begin{proof}
By the definition of $\ve_{\mu,a,s}$ and $\psi_{a,s}$, 
\begin{align*}
\ve_{\mu,a,s}(x,r) &= \frac1{r^s} \biggl[ \int_{|x-y|\leq r} + \sum_{j\geq1} \int_{2^{j-1}r<|x-y|\leq 2^jr} \biggr]\psi_{a,s}
\Bigl(\frac{y-x}r\Bigr)d\mu(y)\\
& \leq C \sum_{j\geq0} \frac{\mu(B(x,2^jr))}{2^{j(s+a)}\,r^s}.
\end{align*}
If $2^jr\leq\delta$, we use the estimate \eqref{eq21}. Otherwise, we take into account that
$$\frac{\mu(B(x,2^jr))}{2^{js}\,r^s} \leq \frac{\mu(\C)}{\delta^s}.$$
So if $N$ denotes the biggest integer such that $2^Nr\leq\delta$, then
\begin{align*}
\ve_{\mu,a,s}(x,r) & \leq C \theta_1 \sum_{0\leq j\leq N} 2^{-ja}+ C 
\frac{\mu(\C)}{\delta^s} \sum_{j\geq N+1} 2^{-ja} \\
& \leq C \biggl(\theta_1 
+ 2^{-Na} \,\frac{\mu(\C)}{\delta^s}\biggr)
\leq C \biggl(\theta_1 
+ \frac{r^a\mu(\C)}{\delta^{s+a}}\biggr).
\end{align*}
If we take $\delta'$ small enough so that 
$$\frac{(\delta')^a\mu(\C)}{\delta^{s+a}} \leq \theta_1,$$
the lemma follows. 
\end{proof}

\medskip
\begin{lemma}
Let $0<t<2$ and set $t'=\frac{2Kt}{2+(K-1)t}$. Let $\phi:\C\to\C$ be a principal $K$-quasiconformal mapping, conformal outside the unit disk, and let $E\subset B(0,1/2)$. Then
\begin{equation}\label{eqhaus2}
\H^{t'}(\phi(E))\leq C(K)\,\H^t(E)^\frac{t'}{Kt}.
\end{equation}
\end{lemma}
\begin{proof}
To prove \rf{eqhaus2}, we may assume that $\H^{t'}(\phi(E))>0$. Because of the estimates on the upper density of Hausdorff measures (see \cite[p.89]{mattila}), there exists  $\delta>0$ and $F\subset E$ compact such that
$\H^{t'}(\phi(F))\geq \H^{t'}(\phi(E))/2$, and
\begin{equation}\label{eqdens}
\frac{\H^{t'}(\phi(E)\cap B(x,r))}{r^{t'}} \leq 5\quad \mbox{for all $x\in\phi(F)$ and $0<r\leq \delta.$}
\end{equation}
Let us denote $\mu=\H^{t'}_{|\phi(F)}$, and consider the associated gauge function
$$h_{\mu,a,t'}(x,r)=r^{t'}\,\ve_{\mu,a,t'}(x,r)=\int\psi_{a,t'}\left(\frac{|x-y|}{r}\right)d\mu(y)$$ 
where $\psi_{a,t'}(r)=\frac1{1+r^{a+t'}}$, for $a>0$ small enough. Recall that, by Lemma \ref{Mmu}, $\mu\leq 2\M^{h_{\mu,a,t'}}$.  Further, we can apply Lemma \ref{lemkgran} to $h_{\mu,a,t'}$ and $\ve_{\mu,a,t'}$ (since they fulfill the required assumptions if $a$ is chosen small enough, by Lemma \ref{lemtec5}), and then we get
\begin{equation}\label{eq2}
\M^{h_{\mu,a,t'}}(\phi(F)) \leq C(K,t) \,\M^{\wt h}(F)^{\frac{t'}{Kt}},
\end{equation}
where $\wt h(x,r)= r^{t}\wt\ve(x,r)$ and
$$\wt\ve(x,r)= \ve_{{\mu,a,t'}}(\phi(B(x,r)))^\frac{Kt}{t'}.$$
In particular, $\M^{\wt h}(F)>0$, and by Frostman's Lemma, there exists a measure $\nu$, supported on $F$, such that
\begin{equation}\label{eq61}
\nu(B(x,r))\leq \wt h(B(x,r)) = r^{t}\ve_{\mu,a,t'}(\phi(B(x,r)))^\frac{Kt}{t'}.
\end{equation}
Furthermore, we can choose $\nu$ so that $\nu(F)\geq C\,\M^{\wt h}(F)$.
It now suffices to show that $\ve_{\mu,a,t'}(\phi(B(x,r)))$ is uniformly bounded for $r$ small enough, as then $\nu(F)\leq C\H^t(F)$.
\\
\\
From \rf{eqdens} and Lemma \ref{lemdens} we infer that $\ve_{\mu,a,t'}(y,s)\leq C_1$ for all $y\in \phi(F)$ and $0<s<\delta'$, with $\delta'=\delta'(\delta, a,t',\mu(F))$, and $C_1=C_1(a,t')$. As a consequence, if $\delta''>0$ is taken small enough, then
\begin{equation}\label{growtheps}
\ve_{\mu,a,t'}(\phi(B(x,r))\leq C_1,\hspace{1cm}\text{for all $x\in F$ and $0<r<\delta''$.}
\end{equation}
To see this, first by quasisymmetry
$$B(\phi(x),r_1)\subset \phi(B(x,r))\leq B(\phi(x),r_2)\hspace{1cm}\text{ for some }r_2\leq C(K) r_1.$$
Moreover, from the local H\"older continuity of $K$-quasiconformal mappings,
$$\frac{2r_1}{\diam\phi(\D)}\leq C(K)\,\left(\frac{r}{\diam\D}\right)^\frac1K,$$
and since $\phi$ is conformal on $\C\setminus\D$, $|\phi(\D)|\leq C(K)\,|\D|$. Hence $r_1\leq C(K)r^{1/K}$, and therefore $\ve_{\mu,a,t'}(\phi(B(x,r))\leq\ve_{\mu,a,t'}(B(\phi(x),r_2))\leq C_1$ whenever $r_2<\delta'$, which immediately follows if $r<\delta''=C(K)\,(\delta')^K$. This proves \eqref{growtheps}.\\
\\
From \eqref{growtheps} and \rf{eq61}, we immediately get that
$$\nu(B(x,r))\leq C_3r^{t}\quad\mbox{all $x\in F$ and $0<r<\delta''$,}$$
with $C_3=C_3(K,t,t')$. It is easy to check that this implies that  
$\H^{t}(F)\gtrsim \nu(F)$.
Indeed, if 
$F\subset \bigcup_i A_i$, with $\diam(A_i)\leq d\leq \delta''$ and $A_i\cap 
F\neq\varnothing$, we take a ball
$B_i$ centered on $F\cap A_i$ with radius $r(B_i)=\diam(A_i)\leq \delta''$ for each $i$, and then
$$\sum_i\diam(A_i)^{t}= \sum_i r(B_i)^{t}\gtrsim \sum_i \nu(B_i) \geq \nu(F),$$
and so $\H^{t}_{d}(\phi(F))\gtrsim \nu(F)$ for all $0<d<\delta''$. Letting $d\to0$ our claim follows.
\\
\\
Summarizing, 
$$\aligned
\H^t(E)\geq \H^t(F)&\geq C\,\nu(F)\geq C\,\M^{\wt h}(F)\geq C\,\M^{h_{\mu,a,t'}}(\phi(F))^\frac{Kt}{t'}\\
&\geq C\,\mu(\phi(F))^\frac{Kt}{t'}\geq C\,\H^{t'}(\phi(F))^\frac{Kt}{t'}\geq C\,\H^{t'}(\phi(E))^\frac{Kt}{t'}.
\endaligned$$
\end{proof}

\medskip
\begin{proof}[\bf Proof of Theorem \ref{teopri}]
The theorem follows from the preceding lemma by standard arguments in quasiconformal theory. However, for completeness we give the details. We factorize $\phi=\phi_2\circ \phi_1$, where $\phi_1$, $\phi_2$ are both $K$-quasiconformal maps, with $\phi_1$ principal and conformal on $\C\setminus 2B$, and
$\phi_2$ is conformal on $\phi_1(2B)$. Let $g(z)=dz+ b$ be the linear function that maps the unit disk to $2B$ (so $d=\diam(B)$). The function $h=g^{-1}\circ \phi_1\circ g$ verifies the assumptions of the main lemma, so that
$$\H^{t'}(g^{-1}\circ\phi_1(E)) \leq C(K) \,\H^t(g^{-1}(E))^{\frac{t'}{Kt}}.$$
On the other hand,
$$\H^{t'}(g^{-1}\circ\phi_1(E)) = \frac{\H^{t'}(\phi_1(E))}{\diam(B)^{t'}},\qquad
\H^t(g^{-1}(E))=\frac{\H^t(E)}{\diam(B)^t}.$$
Using also quasisymmetry and Koebe's distortion theorem, we get that $\diam(\phi_1(B))\simeq\diam(\phi_1(2B))\simeq\diam(2B)$ with constants depending only on $K$. Hence
\begin{equation}\label{eqc4}
\frac{\H^{t'}(\phi(E))}{\diam(\phi_1(B))^{t'}}\leq C(K)\,\left(\frac{\H^t(E)}{\diam(B)^t}\right)^{\frac{t'}{Kt}}.
\end{equation}
Now, since $\phi_2$ is conformal on $\phi_1(2B)$, by Koebe's distortion theorem and quasisymmetry, for each ball $B_0$ contained in $B$ we have 
$$\frac{\diam(\phi_2(\phi_1(B_0)))}{\diam(\phi_2(\phi_1(2B)))}\simeq \frac{\diam(\phi_1(B_0))}{\diam(\phi_1(2B))}.$$
From this estimate and quasisymmetry again, it is straightforward to check that
$$\frac{\H^{t'}(\phi_1(E))}{\diam(\phi_1(B))^{t'}}\simeq \frac{\H^{t'}(\phi(E))}{\diam(\phi(B))^{t'}},$$
with constants depending on $K$, which together with \rf{eqc4} yields \rf{eqhaus}.
\end{proof}

\bigskip



\section{Examples showing sharpness of results}\label{SectionExamples}

In \cite[Thm 2.2]{uriartesharpqcstretching}, an example was constructed of a $K$-quasiconformal map $\phi: \C \to \C$ and a compact set $E \subset \D$ such that $diam(E) \simeq diam (\phi E) \simeq 1$ and such that, for $0<t<2$ and $t' = \frac{2Kt}{2+(K-1)t}$, $\H^t(E) \simeq \H^{t'}(\phi E) \simeq 1$. In the same paper, \cite[Cor 3.5]{uriartesharpqcstretching}, an example is constructed with the same hypotheses, except that both $\H^t(E)$ and $\H^{t'}(\phi E)$ are sigma-finite (but infinite.) These prove the sharpness of Corollary  \ref{teosigma} and Theorem \ref{teopri}.

We will show next that Theorem \ref{teocap2} is sharp. This was already shown for the case $\beta = \frac{2}{3}$, $q = \frac{3}{2}$ in \cite[Thm 8.8]{Tolsa-Uriarte}. We will follow the scheme in \cite{Tolsa-Uriarte} very closely, repeating some of the arguments from \cite{Tolsa-Uriarte} for the convenience of the reader.

\subsection{Basic construction for the subsequent examples}\label{BasicConstructionCantorSets}

Following the scheme of \cite{Tolsa-Uriarte}, we argue as in \cite{uriartesharpqcstretching}. We assume the reader is familiar with the latter paper and we will use the notation from it without further reference. The formulae look slightly nicer if we assume in the construction that $\varepsilon_n = 0$ for all $n$, i.e. that we take infinitely many disks in each step, completely filling the area of the unit disk $\mathbb D$ (see equations (2.1), (2.2) and (2.3) in \cite{uriartesharpqcstretching}.) It is not strictly needed to set in that construction $\varepsilon_n = 0$ for all $n$, and we will later indicate the corresponding formulae if $\varepsilon_n > 0$ for all $n$ (which is the case in \cite{uriartesharpqcstretching}.) The construction in \cite{uriartesharpqcstretching} works as well if we set $\varepsilon_n = 0$ for all $n$, the only point that the reader might wonder about is whether the resulting map is $K$-quasiconformal. 
However, this can be easily seen by a compactness argument (approximating the desired map by maps with finitely many circles in each step which are $K$-quasiconformal and have more and more disks in each step of the construction).

So we get (see equations (2.5) and (2.6) in \cite{uriartesharpqcstretching}) a Cantor type set $E$ and a $K$-quasiconformal map $\phi$ so that a building block in the $N$-th step of the construction of the source set $E$ is a disk with radius
\begin{equation}\label{RadiusSourceNthStep}
s_{j_1,...,j_N}=\left( (\sigma_{1,j_1})^K \, R_{1,j_1} \right) \dots \left( (\sigma_{N,j_N})^K R_{N,j_N} \right) ,
\end{equation}
and a building block in the $N$-th step of the construction in the target set $\phi(E)$ is a disk with radius given by
\begin{equation}\label{RadiusTargetNthStep}
t_{j_1,...,j_N}=\left( \sigma_{1,j_1} \, R_{1,j_1} \right) \dots \left( \sigma_{N,j_N} \, R_{N,j_N} \right) .
\end{equation}

Now we consider a measure $\mu$ supported on $\phi(E)$ (which will be the ``large" set of dimension $t'$) and its image measure $\nu = \phi^{-1}_\ast \mu$ supported on $E$ (which will be the ``small" set of dimension $t$) given by splitting the mass according to area. More explicitly,

\begin{equation}\label{DefinitionOfMuStep0}
\mu (\D) = 1,
\end{equation}
for any disk $B_{1,j_1} = \psi^{i_1}_{1,j_1} \left( \, \overline{\D} \, \right)$ of the first step of the construction with radius $t_{j_1} = \left( \sigma_{1,j_1} \, R_{1,j_1} \right)$,
\begin{equation}\label{DefinitionOfMuStep1}
\mu (B_{1,j_1}) = \left( R_{1,j_1} \right)^2,
\end{equation}
and in general, for any disk $B_{N ; j_1, \dots , j_N}^{i_1, \dots , i_N} = \psi^{i_1}_{1,j_1}  \circ \dots \circ \psi^{i_N}_{N,j_N} \left( \, \overline{\D} \, \right)$ of the $N^{th}$ step of the construction with radius 
$t_{j_1,...,j_N}=\left( \sigma_{1,j_1} \, R_{1,j_1} \right) \dots \left( \sigma_{N,j_N} \, R_{N,j_N} \right) $,
\begin{equation}\label{DefinitionOfMuStepN}
\mu (B_{N ; j_1, \dots , j_N}^{i_1, \dots , i_N}) = \left( R_{1,j_1} \dots R_{N,j_N} \right)^2\ .
\end{equation}

Since we took $\varepsilon_N = 0 $ for all $N$, the total mass of $\mu$ is $1$ in every step. (If one prefers to take $\varepsilon_N > 0 $ for all $N$, the definition should be changed to $\mu (B_{N ; j_1, \dots , j_N}^{i_1, \dots , i_N}) = \left( R_{1,j_1} \dots R_{N,j_N} \right)^2\ \ \prod_{n=N+1}^{\infty} \left( 1- \varepsilon_n \right) $, and the total mass of $\mu$ is renormalized by the factor $\prod_{n=1}^{\infty} \left( 1- \varepsilon_n \right) >0
$, but otherwise the rest of the construction we are about to describe works well.)

Since $\nu$ is the image measure, for any disk $D_{N ; j_1, \dots , j_N}^{i_1, \dots , i_N} = \varphi^{i_1}_{1,j_1}  \circ \dots \circ \varphi^{i_N}_{N,j_N} \left( \, \overline{\D} \, \right) = \phi^{-1}( B_{N ; j_1, \dots , j_N}^{i_1, \dots , i_N} = \varphi^{i_1}_{1,j_1}  \circ \dots \circ \varphi^{i_N}_{N,j_N} \left( \, \overline{\D} \, \right) )$
we get
\begin{equation}\label{DefinitionOfNuStepN}
\nu (D_{N ; j_1, \dots , j_N}^{i_1, \dots , i_N}) = \left( R_{1,j_1} \dots R_{N,j_N} \right)^2\ .
\end{equation}

The following lemma simplifies the computation of the Wolff potentials for the Cantor type sets just described. It was proved in \cite{Tolsa-Uriarte}, but we recall it here (as well as its proof) for the convenience of the reader.

\begin{lemma}\label{ComputingWolffPotentialInOurCantorSets}
For the Cantor type sets just described (in subsection \ref{BasicConstructionCantorSets}), for any $\alpha, p >0$ with $\alpha p < 2$, and for $x \in \phi (E)$, the Wolff potentials satisfy

$$\dot W^\mu_{\alpha,p}(x) \simeq \sum_n \biggl(\frac{\mu(B(x,2^{n}))}{2^{n(2-\alpha p)}}\biggr)^{p'-1}\, \simeq 
\sum_{ N : x \in B_{N ; j_1, \dots , j_N}^{i_1, \dots , i_N} } \biggl(\frac{\mu( B_{N ; j_1, \dots , j_N}^{i_1, \dots , i_N} ) }{ \left( t_{j_1,...,j_N} \right)^{(2-\alpha p)}}\biggr)^{p'-1} \ ,
$$
and analogously for $\nu$, $D_{N ; j_1, \dots , j_N}^{i_1, \dots , i_N}$ and $s_{j_1,...,j_N}$.

\end{lemma}

\begin{proof}

We first introduce some convenient notation. 
For any multiindexes $I=(i_1,...,i_N)$ and  $J=(j_1,...,j_N)$, where $1\leq i_k , 
j_k \leq \infty $ (since we are taking infinitely many disks in each step of the construction), 
we will denote by
\begin{equation}\label{DefinitionProtectingDisk}
P^{N}_{I;J} = \frac{1}{\sigma_{N,j_N}}\, \psi^{i_1}_{1,j_1} \circ \dots \circ \psi^{i_N}_{N,j_N}(\D)
\end{equation}
a {\it{protecting}} disk of generation $N$. Then, $P^N_{I;J}$ has radius 
$$r(P^{N}_{I;J}) = \frac{1}{ \sigma_{N,j_N} } t_{j_1,...,j_N}=\left( \sigma_{1,j_1} \, \dots \sigma_{N-1,j_{N-1}} \right) \left(   R_{1,j_1} \dots   R_{N,j_N} \right).$$ 
Analogously, we will write
\begin{equation}\label{DefinitionGeneratingDisk}
G^{N}_{I;J} = \psi^{i_1}_{1,j_1}  \circ \dots \circ \psi^{i_N}_{N,j_N}(\D)
\end{equation}
to denote a {\it{generating}} disk of generation $N$, which has radius $$r(G^{N}_{I;J}) = t_{j_1,...,j_N}=\left( \sigma_{1,j_1} \, \dots \sigma_{N,j_N} \right)  \left(   R_{1,j_1} \dots   R_{N,j_N} \right).$$

Notice that, since all values of $\sigma_{n,j_n}$ and $R_{n,j_n}$ are $\leq \frac{1}{100}$, then $\mu (G^{N}_{I;J}) = \mu (2 G^{N}_{I;J})$, so we can pretend without loss of generality that the radii $t_{j_1,...,j_N}$ are of the
form $2^k$, $k\in\Z$.

Now, if $ r(G^{N}_{I;J}) \lesssim t \lesssim r(P^{N}_{I;J}) $, and $x \in \phi (E)$ so that $B(x, t) \subseteq P^{N}_{I;J}$, then $\mu ( B(x,t)) = \mu ( G^{N}_{I;J} )$, so that 
$$
\sum_{n : G^{N}_{I;J} \subseteq B(x,2^{n}) \subseteq P^{N}_{I;J}} \biggl(\frac{\mu(B(x,2^{n}))}{2^{n(2-\alpha p)}}\biggr)^{p'-1}
$$
is a geometric series with sum comparable, with constants depending on $\alpha$ and $p$, to its largest term, namely
to $ \biggl(\frac{\mu( G^{N}_{I;J} ) }{ \left( t_{j_1,...,j_N} \right)^{(2-\alpha p)}}\biggr)^{p'-1}  $.

And if $ r(P^{N}_{I;J}) \lesssim t \lesssim r(G^{N-1}_{I';J'}) $, where $G^{N-1}_{I';J'}$ is the unique generating disk of generation $N-1$ containing $P^{N}_{I;J}$, and $x \in \phi (E)$, so that $P^{N}_{I;J} \subseteq B(x, t) \subseteq G^{N-1}_{I';J'}$, then 
\begin{equation}\label{UpperBoundForMuOfBallIfBallInBetweenProtectingNAndGeneratingNMinus1}
\mu ( B(x,t)) \lesssim \frac{ t^2}{ \left( \sigma_{1,j_1} \, \dots \sigma_{N-1,j_{N-1}} \right)  \left(   R_{1,j_1} \dots   R_{N-1,j_{N-1}} \right)}  \left(   R_{1,j_1} \dots   R_{N-1,j_{N-1}} \right)^2 \ , 
\end{equation}
i.e., the mass that $\mu$ assigns to $B(x,t)$ is proportional to its area once $G^{N-1}_{I';J'}$ is renormalized to $\D$, but multiplied by the mass that $\mu$ assigns to $G^{N-1}_{I';J'}$, namely $\left(   R_{1,j_1} \dots   R_{N-1,j_{N-1}} \right)^2$. Hence 
$$
\sum_{n : P^{N}_{I;J} \subseteq B(x,2^{n}) \subseteq G^{N-1}_{I';J'}} \biggl(\frac{\mu(B(x,2^{n}))}{2^{n(2-\alpha p)}}\biggr)^{p'-1} 
$$
is dominated by a geometric series (if $n$ appears in the above sum and $2^{n} = \frac{ r(G^{N-1}_{I';J'}) }{2^k}$ with $k >0$, then $$ \biggl(\frac{\mu(B(x,2^{n}))}{2^{n(2-\alpha p)}}\biggr)^{p'-1} \lesssim  \biggl(\frac{\mu( G^{N-1}_{I';J'})}{r(G^{N-1}_{I';J'})^{(2-\alpha p)}} \ \frac{ 2^{k(2-\alpha p)} }{2^{2k} }  \biggr)^{p'-1},$$  and hence the above sum is $ \lesssim  \biggl(\frac{\mu( G^{N-1}_{I';J'} )}{r(G^{N-1}_{I';J'})^{(2-\alpha p)}}\biggr)^{p'-1}$, with constants depending only on $\alpha$ and $p$.)
\end{proof}

\subsection{Example}\label{CantorSetExample}

In order to see that Theorem \ref{teocap2} is sharp, it is useful to recall Theorem 5.5.1 (b) in \cite{adamshedberg} adapted to our situation (and combined with Proposition 5.1.4):

\begin{theorem}\label{TheoremRelatingCapacities}
Let $E \subset \C$. Then there is a constant $A$ such that 
$$
\dot C_{\beta, q} (E) \leq A \dot C_{\alpha, p} (E) \ ,
$$
for $\beta q = \alpha p = 2-\frac{2}{K+1} = \frac{2K}{K+1}$, $p<q$.

Moreover, there exist sets $E$ such that $\dot C_{\beta, q} (E) =0$ but $\dot C_{\alpha, p} (E) >0$.
\end{theorem} 

Hence it is conceivable that Theorem \ref{teocap2} might be strengthened to a statement of the form 
$$
\frac{\dot{C}_{\beta,q}(\phi(E))}{\diam(\phi(B))^{t'}}\leq C(\beta,q,K)\,\left(\frac{\dot{C}_{\widetilde{\alpha}, \widetilde{p}}(E)}{\diam(B)^t}\right)^\frac{t'}{Kt} \;,
$$
for some $\widetilde{\alpha}, \widetilde{p}$ such that $\widetilde{\alpha} \widetilde{p} = \alpha p = 2-t$ and $\widetilde{p} > p$. The following theorem shows that the answer to this question is negative.

\begin{theorem}\label{teosharp}
For any $\widetilde{\alpha}, \widetilde{p} >0$ such that $\widetilde{\alpha} \widetilde{p} = \alpha p = 2-t$ and $ \widetilde{p} > p$, there exists a compact $E \subset \C$ and a $K$-quasiconformal map $\phi$ such that $\dot C_{ \beta, q}(\phi E) >0$ (and hence $\dot C_{ \alpha , p }(E) >0$, due to Theorem \ref{teocap2}), but $ \dot C_{ \widetilde{\alpha}, \widetilde{p}}(E) = 0$.
\end{theorem}

\begin{proof}

For $E$ and $\phi$ as in Subsection \ref{BasicConstructionCantorSets}, 
notice that by Lemma \ref{ComputingWolffPotentialInOurCantorSets}, for $x \in \phi (E)$

\begin{eqnarray*}
\dot W^{\mu}_{ \beta, q } (x) & \simeq & \sum_{ N : x \in B_{N ; j_1, \dots , j_N}^{i_1, \dots , i_N} } \biggl(\frac{\mu( B_{N ; j_1, \dots , j_N}^{i_1, \dots , i_N} ) }{ \left( t_{j_1,...,j_N} \right)^{2-\beta q} } \biggr)^{q'-1}  \\
& = & \sum_{ N : x \in B_{N ; j_1, \dots , j_N}^{i_1, \dots , i_N} } \biggl(\frac{ \left( R_{1,j_1} \dots R_{N,j_N} \right)^2 }{  \left(  \sigma_{1,j_1} \dots \sigma_{N,j_N} R_{1,j_1} \dots R_{N,j_N} \right)^{2-\beta q} }\biggr)^{q'-1}.
\end{eqnarray*}

Since on the one hand $E$ is very ``close" to satisfying $0 < \H^{t}(E) < \infty$ 
and $0 < \H^{t'}(\phi E) < \infty$ (see (3.9), (3.10) and (4.5) in \cite{uriartesharpqcstretching}) and, on the other hand, an important element in 
the proof of the semiadditivity of analytic capacity is that the potential is ``approximately 
constant" on each scale (see \cite{tolsasemiadditivityanalyticcapacity}), the above equation 
suggests the choice
\begin{equation}\label{ChoiceOfSigma}
\sigma_{N,j_N} = \left( R_{N,j_N}  \right)^{\frac{2-t}{tK}}  \, d_N\quad\mbox{ for all $N$,}
\end{equation} 
where $d_N \in [1,2]$ is a parameter to be determined, independent of $j_N$.

If we take
\begin{equation}\label{ChoiceOfDjForExample2}
d_j = \left(  \frac{j+1}{j} \right)^\delta,
\end{equation} 
for an appropriate $\delta >0$ to be chosen later, then for $x \in \phi E$, we have
\begin{equation}\label{FirstConditionForDelta}
\dot W^{\mu}_{ \beta, q  } (x) \simeq \sum_n \left\{ \prod_{j=1}^{n} \frac{1}{\left( d_j \right)^{t' (q' -1)}} \right\} = \sum_{n=2}^{\infty} \frac{1}{n^{t' (q' -1) \delta}} < \infty,
\end{equation}
once $\delta >0$ is appropriately chosen, so that $ \dot C_{ \beta, q } (\phi E) >0$ (and hence $\dot C_{ \alpha , p }(E) >0$, due to Theorem \ref{teocap2}.) 

By Lemma \ref{ComputingWolffPotentialInOurCantorSets} and
\rf{RadiusSourceNthStep}, for $x \in E$,
\begin{eqnarray*}
\dot W^{\nu}_{ \widetilde{\alpha}, \widetilde{p} } (x) & \simeq & \sum_{ N : x \in D_{N ; j_1, \dots , j_N}^{i_1, \dots , i_N} } \biggl(\frac{\nu( D_{N ; j_1, \dots , j_N}^{i_1, \dots , i_N} ) }{ \left( s_{j_1,...,j_N} \right)^{ 2- \widetilde{\alpha} \widetilde{p} }}\biggr)^{\widetilde{p}' - 1}  \\
& = & \sum_{ N : x \in D_{N ; j_1, \dots , j_N}^{i_1, \dots , i_N} } \biggl(\frac{ \left( R_{1,j_1} \dots R_{N,j_N} \right)^2 }{ \left[ \left( \sigma_{1,j_1} \dots \sigma_{N,j_N} \right)^{K}  \left( R_{1,j_1} \dots R_{N,j_N} \right) \right]^{ 2- \widetilde{\alpha} \widetilde{p} }
   }\biggr)^{\left( \widetilde{p}'-1 \right) } \ ,
\end{eqnarray*}
so that, substituting $\sigma_{N,j_N} = \left( R_{N,j_N}  \right)^{\frac{2-t}{tK}}  \, d_N$  and $d_j = \left(  \frac{j+1}{j} \right)^\delta$ we get, for $x \in E$,
\begin{equation}\label{SecondConditionForDelta}
\dot W^{\nu}_{ \widetilde{\alpha}, \widetilde{p} } (x) \simeq \sum_n \left\{ \prod_{j=1}^{n} \frac{1}{\left( d_j \right)^{tK}} \right\}^{\left( \widetilde{p}'-1 \right)  } =
\sum_{n=2}^{\infty} \frac{1}{n^{t K (\widetilde{p}' -1) \delta}} = \infty,
\end{equation}
once $\delta >0$ is appropriately chosen, so that $ \dot C_{ \widetilde{\alpha}, \widetilde{p} } ( E) =0$.

In order that both \eqref{FirstConditionForDelta} and \eqref{SecondConditionForDelta} be satisfied, it is enough to choose $\delta > 0$ so that $ t K (\widetilde{p}' -1) \delta \leq 1 < t K (p' - 1) \delta $ (since $t K (p' - 1) = t' (q' -1)$, which by the way implies that, in our example, $ \dot W^{\mu}_{ \beta, q } ( \phi x) \simeq  \dot W^{\nu}_{\alpha, p } (x) $ for $x \in E$.) Since $\widetilde{p} > p $, this can be achieved by setting 
$$
\delta = \frac{ 1}{ t K (\widetilde{p}' -1) } \;.
$$

\end{proof}

Let us remark that the above example also gives that $\dot C_{\gamma,r}(E)=0$ if 
$\gamma\,r<\widetilde{\alpha}\,\widetilde{p} = \alpha\,p = 2-t$. This due to the fact that there is some constant $A$ 
independent of $E$ such that
$$\dot C_{\gamma,r}(E)^{1/(2-\gamma r)}\leq A \,\dot C_{\alpha,p}(E)^{1/(2-\alpha p)}.$$
See Theorem 5.5.1 of \cite{adamshedberg}.

\bigskip
\bigskip

\noindent{\bf Acknowledgements:}
Part of this paper was done while all authors were attending the research semester
``Harmonic Analysis, Geometric Measure Theory and Quasiconformal Mappings'' in
the CRM (Barcelona) in 2009.

\noindent 
K.A.\ is supported by Academy of Finland, grants SA-134757 
and CoE in Analysis and dynamics research, grant SA-18634, and EU-network CODY, grant 
400630.
A.C.\ and J.V.\
are partially supported by grants  2009-SGR-420 (Catalonia) and MTM2007-60062 (Spain). A.C. is also partially supported by EU-network CODY.
X.T.\ is partially supported by grants  2009-SGR-420 (Catalonia) and MTM2007-62817 (Spain). I.U.\ 
was a postdoctoral fellow at the University of Missouri, Columbia, USA, and at Centre de Recerca Matem\`{a}tica, Barcelona, Spain, for some periods of time during the elaboration of this paper. He was partially supported by NSF grant DMS-0901524.

\bigskip

\bibliographystyle{alpha}
\bibliography{./references25G}

\newcommand{\etalchar}[1]{$^{#1}$}
\def\cprime{$'$} \def\cprime{$'$} \def\cprime{$'$}
\begin{thebibliography}{ACM{\etalchar{+}}08}

\bibitem[ACM{\etalchar{+}}08]{ACMOU}
Kari Astala, Albert Clop, Joan Mateu, Joan Orobitg, and Ignacio Uriarte-Tuero.
\newblock Distortion of {H}ausdorff measures and improved {P}ainlev\'{e}
  removability for bounded quasiregular mappings.
\newblock {\em Duke Math. J.}, 141(3):539--571, 2008.

\bibitem[AH96]{adamshedberg}
David~R. Adams and Lars~Inge Hedberg.
\newblock {\em Function spaces and potential theory}, volume 314 of {\em
  Grundlehren der Mathematischen Wissenschaften [Fundamental Principles of
  Mathematical Sciences]}.
\newblock Springer-Verlag, Berlin, 1996.

\bibitem[AIM09]{astalaiwaniecmartin}
Kari Astala, Tadeusz Iwaniec, and Gaven Martin.
\newblock {\em Elliptic partial differential equations and quasiconformal
  mappings in the plane}, volume~48 of {\em Princeton Mathematical Series}.
\newblock Princeton University Press, Princeton, NJ, 2009.

\bibitem[AN03]{astalanesi}
Kari Astala and Vincenzo Nesi.
\newblock Composites and quasiconformal mappings: new optimal bounds in two
  dimensions.
\newblock {\em Calc. Var. Partial Differential Equations}, 18(4):335--355,
  2003.

\bibitem[Ast94]{astalaareadistortion}
Kari Astala.
\newblock Area distortion of quasiconformal mappings.
\newblock {\em Acta Math.}, 173(1):37--60, 1994.

\bibitem[Dav98]{davidunrectifiable1setszeroanalyticcapacity}
Guy David.
\newblock Unrectifiable {$1$}-sets have vanishing analytic capacity.
\newblock {\em Rev. Mat. Iberoamericana}, 14(2):369--479, 1998.

\bibitem[LSUT]{Lacey-Sawyer-Uriarte}
Michael~T. Lacey, Eric~T. Sawyer, and Ignacio Uriarte-Tuero.
\newblock Astala's conjecture on distortion of {H}ausdorff measures under
  quasiconformal maps in the plane.
\newblock {\em To appear in Acta Math.}

\bibitem[Mat95]{mattila}
Pertti Mattila.
\newblock {\em Geometry of sets and measures in {E}uclidean spaces}, volume~44
  of {\em Cambridge Studies in Advanced Mathematics}.
\newblock Cambridge University Press, Cambridge, 1995.
\newblock Fractals and rectifiability.

\bibitem[SS62]{sionsjerve}
M.~Sion and D.~Sjerve.
\newblock Approximation properties of measures generated by continuous set
  functions.
\newblock {\em Mathematika}, 9:145--156, 1962.

\bibitem[Tol03]{tolsasemiadditivityanalyticcapacity}
Xavier Tolsa.
\newblock Painlev\'e's problem and the semiadditivity of analytic capacity.
\newblock {\em Acta Math.}, 190(1):105--149, 2003.

\bibitem[Tol09]{Tolsa-preprintqc}
Xavier Tolsa.
\newblock Quasiconformal distortion of {H}ausdorff measures.
\newblock {\em arXiv:0907.4933v2}, 2009.

\bibitem[TUT09]{Tolsa-Uriarte}
Xavier Tolsa and Ignacio Uriarte-Tuero.
\newblock Quasiconformal maps, analytic capacity, and non linear potentials.
\newblock {\em Preprint}, 2009.

\bibitem[UT08]{uriartesharpqcstretching}
Ignacio Uriarte-Tuero.
\newblock Sharp examples for planar quasiconformal distortion of {H}ausdorff
  measures and removability.
\newblock {\em International Mathematics Research Notices}, 2008: rnn047-43,
  2008.

\end{thebibliography}

\enlargethispage{2cm}
\end{document}